\documentclass[12pt]{amsart}
\usepackage{bbm, amsmath, amsfonts, amscd, latexsym, amsthm, amssymb, graphicx,hyperref,mathrsfs, comment}
\usepackage{abstract}
\usepackage{transparent}
\usepackage{mathtools}
\usepackage{caption}
\usepackage[dvipsnames]{xcolor}

\usepackage{tikz-cd}

\usepackage[usenames,dvipsnames]{pstricks}

\usepackage{graphicx}
\usepackage{color}
\usepackage[all]{xy}
\usepackage[tikz]{tablvar}

\newcommand{\wD}{\widetilde{\Delta}}
\newcommand{\bC}{\mathbb{C}}
\newcommand{\bR}{\mathbb{R}}

\newcommand{\bZ}{\mathbb{Z}}
\newcommand{\bN}{\mathbb{N}}

\newcommand{\ttor}{(\bC^\ast)^2}
\newcommand{\eps}{\varepsilon}
\newcommand{\s}[1]{\scalebox{0.7}{$#1$}}
\newcommand{\tgf}{\widetilde{\gamma}_{\s{f}}}
\newcommand{\tgg}{\widetilde{\gamma}_{\s{g}}}
\newcommand{\gf}{\gamma_{\s{f}}}
\renewcommand{\gg}{\gamma_{\s{g}}}
\renewcommand{\Im}{\text{Im}}
\renewcommand{\Re}{\text{Re}}
\newcommand{\tw}{\tilde w}

\newcommand{\tft}{\tilde{f}_t}
\newcommand{\wN}{\widetilde{N}}
\renewcommand{\ss}[1]{\scalebox{0.7}{$#1$}}

\renewcommand{\emph}[1]{{\it {\color{NavyBlue} #1}}}

\DeclareMathOperator{\Log}{Log}

\DeclareMathOperator{\itr}{int}
\DeclareMathOperator{\conv}{conv}
\DeclareMathOperator{\vol}{vol}

\DeclareMathOperator{\Sym}{Sym}

\DeclareMathOperator{\Cr}{cr}

\DeclareMathOperator{\New}{New}

\DeclareMathOperator{\ini}{in}

\newtheorem{Proposition}{Proposition}[section]
\newtheorem{thm}[Proposition]{Theorem}
\newtheorem{definition}[Proposition]{Definition}
\newtheorem{Lemma}[Proposition]{Lemma}
\newtheorem{Remark}[Proposition]{Remark}
\newtheorem{theorem}{Theorem}

\newtheorem{Corollary}[Proposition]{Corollary}
\newtheorem{Example}{Example}

\title{Patchworking the Log-critical locus of planar curves}
\author{Lionel Lang and Arthur Renaudineau}

\address[Lang]{Faculty of Engineering and Sustainable Developement \\ University of Gävle \\ 80176 \\ Gävle \\ Sweden}
\email{lionel.lang@hig.se}

\address[Renaudineau]{Univ. Lille \\ CNRS \\ UMR 8524 - Laboratoire Paul Painlev\'{e}\\  F-59000 Lille\\ France.}
\email{arthur.renaudineau@univ-lille.fr}

\thanks{The authors are grateful to Jens Forsg{\aa}rd and especially to Timur Sadykov and Dmitrii Bogdanov for providing the pictures based on the work \cite{BKS} that we used for Figure \ref{fig:intro4}. A.R. acknowledges support from the Labex CEMPI (ANR-11-LABX-0007-01).}

\begin{document}

\maketitle

\begin{abstract}
    We establish a patchworking theorem \textit{\`{a} la} Viro for the Log-critical locus of algebraic curves in $\ttor$. As an application, we prove the existence of projective curves of arbitrary degree with smooth connected Log-critical locus. To prove our patchworking theorem, we study the behaviour of Log-inflection points along families of curves defined by Viro polynomials. In particular, we prove a generalisation of a theorem of Mikhalkin and the second author on the tropical limit of Log-inflection points.
\end{abstract}

\section{Introduction}\label{sec:intro}

A fundamental problem in real algebraic geometry is the topological classification of projective non-singular varieties of a fixed degree. For plane curves $\bR C \subset \bR P^2$, the topological classification is known since the work of Harnack (\cite{Har}), but the ambient classification, that is the classification of topological pairs $(\bR C, \bR P^2)$, is still open from degree $8$. \smallskip

A different but related classification problem is the one introduced in \cite{L19}: for a given degree $d$, classify the topological pairs $(C, \Cr(C))$ where $C \subset \bC P^2$ is a curve of degree $d$ and $\Cr(C)\subset C$ is the \emph{Log-critical locus}. Recall that the Log-critical locus is the closure in $C$ of the critical locus of the map 
\[
\begin{array}{rcl}
     \Log : C \cap \ttor & \rightarrow & \bR^2  \\
     (z,w)& \mapsto & \big( \log\vert z \vert , \log \vert w \vert \big) 
\end{array} .
\]
Alternatively, the set $\Cr(C)\subset C$ is the critical locus of the restriction to $C$ of the \emph{moment map} $\mu:X\rightarrow \Delta_d$, where $\Delta_d$ is the standard $d$-simplex.
Important related objects are the \emph{amoeba} $\mu(C)$ and its \emph{contour} $\mu(\Cr(C))$.\smallskip

In analogy to the real locus of a real algebraic curve, the set $\Cr(C)$ is a smooth manifold of dimension 1, that is a disjoint union of ovals in the Riemann surface $C$, provided that $C$ is generic. More precisely, the set of curves $C$ for which $\Cr(C)$ is singular is a semi-algebraic subset of real codimension $1$ in the set of all complex curves of degree $d$, see \cite{L19}. The complement to this discriminantal set turns out to be disconnected for $d\geqslant 2$, allowing distinct topological pairs $(C, \Cr(C))$ for a given degree. \smallskip

There exist strong connections between the ambient classification problem in real algebraic geometry and the classification of Log-critical loci. 
Indeed, it was observed by Mikhalkin in \cite{Mikh} that for a real curve $C$, the real part $\bR C \subset C$ is a subset of $\Cr(C)$. The inclusion $\bR C \subset \Cr(C)$ taking place in $C$ gives valuable information on the inclusion $\bR C\subset \bR P^2$. A striking example is given by the simple Harnack curves introduced in \cite{Mikh}: if the inclusion $\bR C \subset \Cr(C)$ is an equality, then the topological pair $(\bR C, \bR P^2)$ is uniquely determined by the degree. We also refer to \cite{L2} for further results in this direction. \smallskip

Up to now, not much is known on the classification of pairs $(C, \Cr(C))$ for a given degree $d$. As observed in \cite{Mikh}, the set $\Cr(C)$ is the pullback  of $\bR P^1$ under the \emph{logarithmic Gauss map} $\gamma : C \rightarrow \bC P^1$. This implies that the number of connected components of $\Cr(C)$ satisfies
\[1 \leqslant b_0( \Cr(C)) \leqslant d^2\]
where $d^2$ is the degree of $\gamma$. There are very few general constructions of curves $C$ with prescribed values of $b_0( \Cr(C))$. The first example comes from the simple Harnack curves of \cite{Mikh}: for such curves, the set $\Cr(C)$ has $\left(\begin{smallmatrix} d-1\\2\end{smallmatrix}\right)+1$ connected components. 
In \cite{L19}, it was shown that $b_0( \Cr(C))$ can achieve all values between $\left(\begin{smallmatrix} d+1\\2\end{smallmatrix}\right)$ and $d^2$. Besides the latter results, there is no general construction of pairs $(C,\Cr(C))$. In particular, there was no evidence up to now of the existence of a curve $C$ with smooth Log-critical locus satisfying $b_0( \Cr(C)) \leqslant \left(\begin{smallmatrix} d-1\\2\end{smallmatrix}\right)$.\smallskip

Coming back to the classification of real algebraic curves, the most powerful tool to construct pairs $(\bR C, \bR P^2)$ is the Patchworking Theorem introduced by O. Viro, see \cite{Viro84}. In its full generality,  this theorem allows one to construct real algebraic hypersurfaces of a given degree in toric varieties  by gluing a collection of hypersurfaces of smaller degrees using so-called \emph{Viro polynomials}, see for instance the appendix of \cite{GLSV} or \cite{Risler92}.\medskip

The main goal of the present paper is to prove an analogue of Viro's Patchworking Theorem for Log-critical loci and study its applications to the associated classification problem. We obtain a patchworking theorem in the name of Theorem \ref{thm:patchwork} that we use to construct curves whose Log-critical locus has a small number of connected components, see Theorem \ref{thm:applicationpatchwork}. In particular, we show that for any degree $d\geqslant 2$, there exists a smooth curve $C$ such that $\Cr(C)$ is smooth and connected. In particular, we disprove Conjecture 1 in \cite{L19}.\smallskip

The statement of Theorem \ref{thm:patchwork} requires some technicalities and is therefore postponed to Section \ref{sec:main}. In the meantime, let us illustrate the latter results with some examples. To begin with, let us consider the patchworking of the real algebraic curves $C_1$ and $C_2$ defined respectively by the polynomials
\[f_1(z,w)=w(z+2.6)+(z+2.5)(z+1) \quad \text{and} \quad f_2(z,w)=w(z+2.6)+w^2.\]
Here, we denote $\Delta_j$ the Newton polygon of $f_j$ and $\mu_j$ the corresponding moment map, $j=1,2$. The real part $\bR C_j$ can be represented in 4 symmetric copies of $\Delta_j$, one copy per quadrant of $(\bR^*)^2$. This representation, usually referred to as a \emph{chart} of the curve, can be achieved by using an unfolding of $\mu_j$ that remember the signs of each quadrant, see Figure \ref{fig:intro1}.
Since the polynomials $f_1$ and $f_2$ agree with each other on the common edge $\Delta_1\cap\Delta_2$, the Patchworking Theorem states that there exists a curve $C$ with Newton polygon $\Delta_1\cup\Delta_2$ and whose chart is isotopic to the gluing of the charts of $C_1$ and $C_2$, see again Figure \ref{fig:intro1}. Moreover, the curve $C$ can be defined as the zero set of the Viro polynomial
\[f_t(z,w)=w(z+2.6)+(z+2.5)(z+1)+tw^2\]
for $t>0$ small enough. \smallskip
\begin{figure}[h]
    \centering
    \scalebox{1}{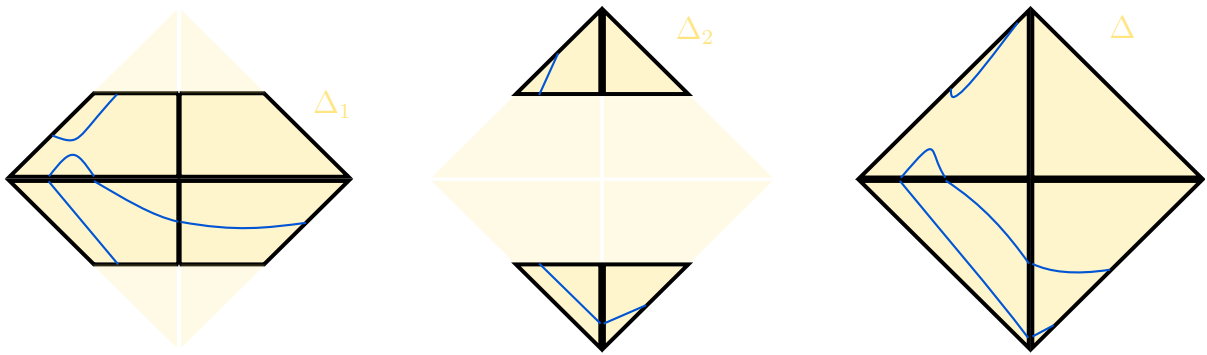}
    \caption{The charts of $\bR C_1$ (left), $\bR C_2$ (middle) and $\bR C$ (right).}
    \label{fig:intro1} 
\end{figure}

There is no direct generalisation of the charts of Figure \ref{fig:intro1} if we perturb the polynomials $f_1$ and $f_2$ in a complex direction since the real parts of $C_1$ and $C_2$ cease to exist. To fix this, we can first fold the 4 copies of each Newton polygon to a single one, 
considering the set $\mu_j(\bR C_j)$ instead of the chart of $C_j$, see Figure \ref{fig:intro2}. Secondly, we can trade the set $\mu_j(\bR C_j)$ for its superset $\mu_j(\Cr(C_j))$, that is the contour of $C_j$. We can use such kind of representations for an arbitrary patchwork. This provides an indirect picture on how the Log-critical loci of the various curves involved in the patchwork eventually glue together along the families of curves defined by the underlying Viro polynomial.
\begin{figure}[b]
    \centering
    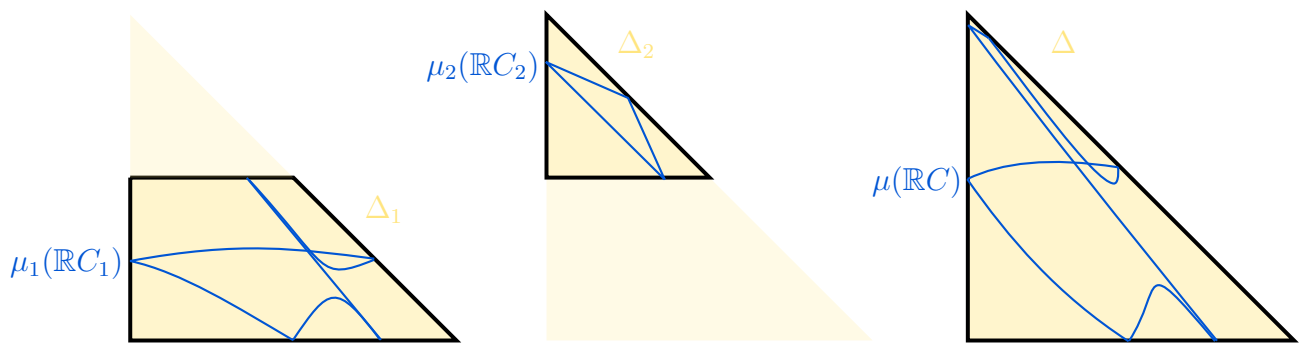
    \caption{The amoeba of $\bR C$ (right) obtained as the patchwork of the amoebas of $\bR C_1$ (left) and $\bR C_2$ (middle).}
    \label{fig:intro2}
\end{figure}

Instead of the polynomials $f_1$, $f_2$ and $f_t$ used above, consider now
\[\tilde f_1(z,w)=w(z+2.6+0.5 i)+(z+2.5+0.5 i)(z+1+0.5 i) \, , \; \; \tilde f_2(z,w)=w(z+2.6+0.5 i)+w^2\]
and the Viro polynomial
\[\tilde f_t(z,w)=w(z+2.6+0.5 i)+(z+2.5+0.5 i)(z+1+0.5 i)+tw^2\]
and add a tilde to every piece of notation. Then, 
Theorem \ref{thm:patchwork} describe the set $\Cr(\widetilde C)$ for a generic choice of the complex parameter $t$, provided that $\vert t \vert$ is small.  In Figure \ref{fig:intro3}, we picture the patchworking procedure at the level of contour as well as the amoeba $\Log(\widetilde C)$ for the parameter $t=0.003$. 
In the present case, Theorem \ref{thm:patchwork} implies that $b_0\big(\Cr(\widetilde C)\big)=2$, a fact that can be read from Figure \ref{fig:intro3}. To our knowledge, this is the first instance of a curve whose Newton polygon is the 2-simplex and such that its critical locus is smooth and has exactly 2 connected components. An additional patchwork using the latter curve leads to a cubic curve $\widehat C$ whose Log-critical locus is smooth and connected. Concretely, if $\tilde f$ is the defining polynomial of $\widetilde C$, then the curve $\widehat C$ can be defined by the polynomial $\hat f=z \tilde f -8\cdot 10^{-7} (w+1)(w+10)(w+100)$, see Firgure \ref{fig:intro5}.
%Here, it is worth noticing that, although the pictures in the center and to the right represent the same object, they reveal different properties. In general, it is challenging to obtain readable picture, partly because the Patchworking Theorem may only hold for very small values of $t$.
\smallskip
\begin{figure}[h]
    \centering
    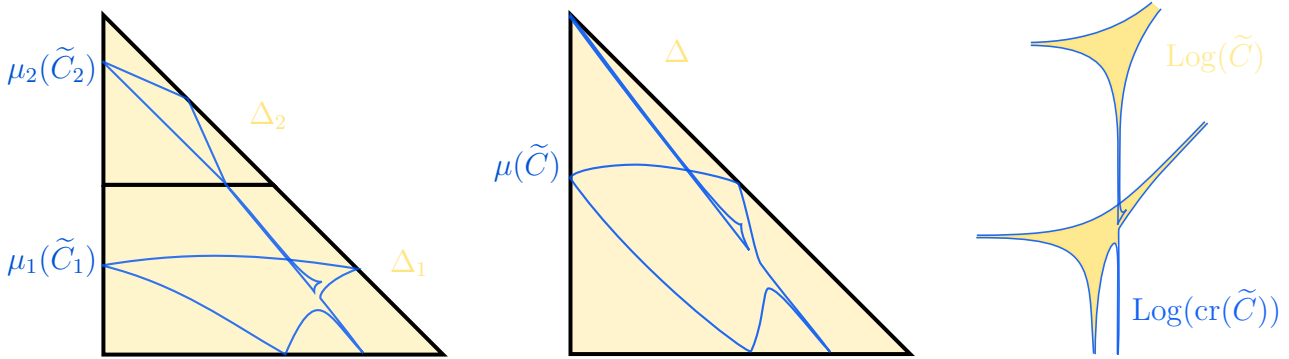
    \caption{The contour of $\widetilde C$ (middle) obtained as the patchwork of the contours of $\widetilde C_1$ and $\widetilde C_2$ (left). On the right, we represent the non-compact amoeba $\Log(\widetilde C)$ together with its contour.}
    \label{fig:intro3}
\end{figure}

The efficiency of Theorem \ref{thm:patchwork} for constructing curves with prescribed Log-critical locus relies on the variety of building blocks that are already at our disposal. In Section \ref{sec:auxilliarycurves}, we provide a collection of such building blocks that we use later on to prove Theorem \ref{thm:applicationpatchwork}.  Those blocks are curves in the Hirzebruch surface $\Sigma_1$. If we denote $(z,w)$ the coordinates of $\ttor\subset \Sigma_1$, these curves intersect the divisor $z=0$ exactly once. More importantly, they have the remarkable property that each intersection point with the divisor $w=\infty$ is contained in a single component of the Log-critical locus. The latter phenomenon can be observed in Figure \ref{fig:intro4} where the contour of such curves is depicted. There, we can observe the ``birth'' of each of the tentacles going upwards, involving exactly one component of the Log-critical locus.

\begin{figure}[t]
    \centering
    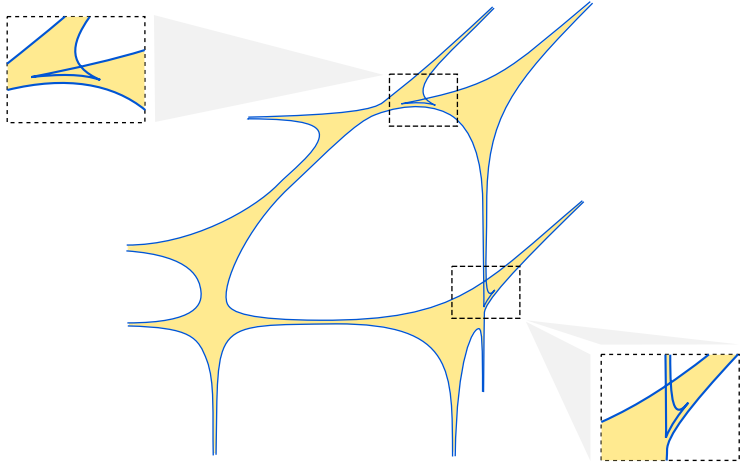
    \caption{The amoeba (yellow) and the contour (blue) of the curve $\widehat C$.}
    \label{fig:intro5}
\end{figure}

To obtain Theorem \ref{thm:patchwork}, we study the behaviour of \emph{Log-inflection points} along families of curves defined by Viro polynomials. Recall that a Log-inflection point is a ramification point for the the Logarithmic Gauss map $\gamma$. The relative position with $\bR P^1$ of the corresponding branching points in $\bC P^1$ governs the topology of the pair $(C,\Cr(C))$ since $\Cr(C)=\gamma^{-1}(\bR P^1)$. 
While patchworking a collection of curves $C_1,\; \cdots ,\; C_k$, the 1-parameter family of curves defined by the underlying Viro polynomial can be compactified with the reducible curve $C_1\cup \cdots \cup C_k$ at $t=0$, using an appropriate toric $3$-fold. When $t$ tends to $0$, some of the Log-inflection points of the generic curve concentrate at the nodal points of the central curve $C_1\cup \cdots \cup C_k$. We describe the asymptotic of these points in Theorem \ref{thm:asymptotic}, see Section \ref{sec:proof}. The latter result is the cornerstone of the proof of Theorem \ref{thm:patchwork}. It also leads to a generalisation of a theorem of Mikhalkin and the second author (see \cite{MikhRena}) in the name of Theorem \ref{thm:tropicallimit}. The latter theorem describes the tropical limit of the Log-inflection points along families of curves defined by generic Viro polynomials. The original statement \cite[Theorem 3]{MikhRena} asserts that Log-inflection points accumulate by pairs at the midpoint of every bounded edge of a tropical curve, if the tropical curve is non-singular. We prove that the same phenomenon occurs for families of curves defined by generic Viro polynomials associated to arbitrary subdivisions. In particular, the limiting tropical curve may be singular in this context. Eventually, we show in the appendix that the genericity assumption on the Viro polynomial is necessary. To do so, we exhibit Viro polynomials for which the tropical limit of some Log-inflection points is located at 1/3 of a bounded edge of the tropical limit. \newpage

\begin{figure}[t]
    \centering
    \includegraphics[scale=1.2]{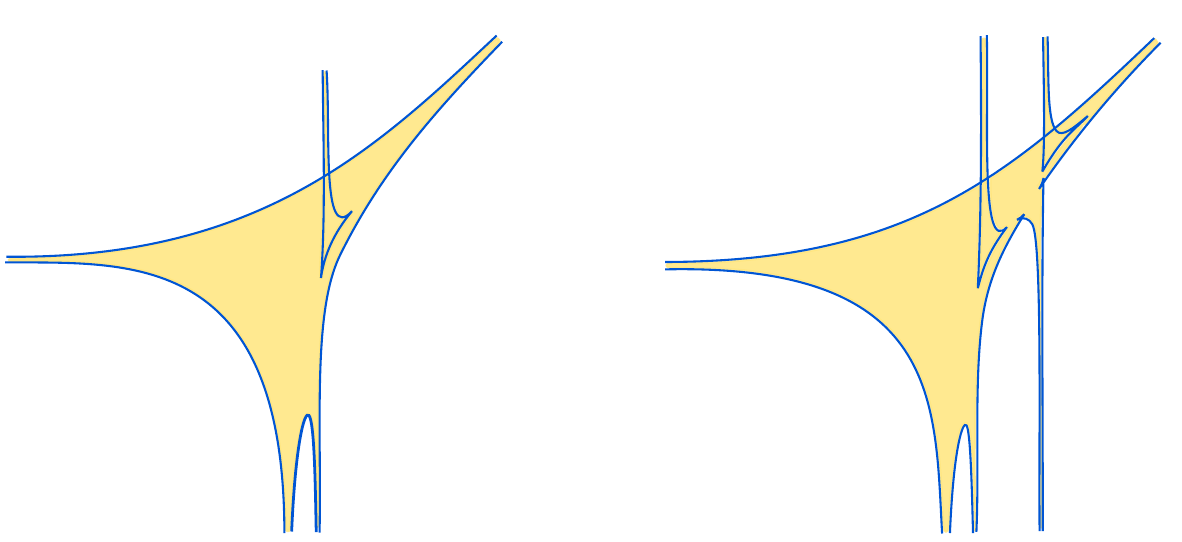}
    \caption{The amoeba (yellow) and the contour (blue) of the polynomials $w(z+17-5i)+(z+15-5i)(z+10-5i)
$ (left) and $w (z+37-5i)(z+17-5i)+(z+35-5i)(z+15-5i)(z+10-5i)$ (right).}
    \label{fig:intro4}
\end{figure}

\tableofcontents

\section{Setting} \label{sec:setting}
\subsection{Viro polynomial}\label{sec:viropol}

Throughout this text, the symbol $\Delta$ refers to a \emph{lattice polygon}, that is the convex hull in $\bR^2$ of a finite set of points in $\bZ^2$. We denote by $\emph{\New(f)}$ the Newton polyhedron of any Laurent polynomial $f$. We denote by \emph{$(X_\Delta, \mathcal{L}_\Delta)$} the polarized \emph{toric surface} associated to $\Delta$, see \cite[Chapter 5]{GKZ}. The monomial embedding (1.2) of the latter reference allows to identify the space of section $\vert \mathcal{L}_\Delta \vert$ with the projectivisation of the space of Laurent polynomials whose Newton polygon is contained in $\Delta$. The toric surface $X_\Delta$ provides a compactification of $\ttor$ with a chain of toric divisors. Each such divisor is isomorphic to $\bC P^1$ and correspond to an edge of $\Delta$ via the moment map.

\begin{definition}\label{def:div}
A \emph{subdivision} of a lattice polygon $\Delta$ is a set of lattice polygons $\{\Delta_k\}_{k\in I}$ such that:

\begin{itemize}
\item $\cup_{k\in I}\Delta_k=\Delta$,
\item if $k,l\in I$, then the intersection $\Delta_k\cap\Delta_l$ is a common face to $\Delta_k$ and $\Delta_l$.
\end{itemize}
A subdivision $\{\Delta_k\}_{k\in I}$ of $\Delta$ is said to be \emph{convex} if there exists a convex piecewise-linear function $\nu:\Delta\rightarrow\bR$ whose domains of linearity coincide with the polygons $\Delta_k$. If such a function exists, then it can be taken in such a way that $\nu(\Delta\cap\bZ^2)\subset \bZ^2$. We call such a function \emph{$\bZ$-convex}. 
\end{definition}

For any $\bZ$-convex function $\nu:\Delta\rightarrow\bR$ and any Laurent polynomial
\[
f(z,w)=\sum_{(i,j)\in\Delta\cap\bZ^2}a_{i,j}z^iw^j,
\]
we define the associated \emph{Viro polynomial} to be the Laurent polynomial
\[
f_t(z,w):= f_{t,\nu}(z,w):=\sum_{(i,j)\in\Delta\cap\bZ^2}a_{i,j}t^{\nu(i,j)}z^iw^j.
\]
For any $k\in I$, we denote by \emph{$f^{\Delta_k}$}$:=\sum_{(i,j)\in\Delta_k\cap\bZ^2}a_{i,j}z^iw^j$ the truncation of $f$ to $\Delta_k$.
The Viro polynomial $f_t$ is said to be \emph{non-degenerate} if for any $k\in I$, the curve $Z(f^{\Delta_k})$ is smooth and transverse to every toric divisor of $X_{\Delta_k}$.

Recall that the function $\nu$ induces a toric compactification of $(\bC^*)^3$ in which all the curves $Z(f_t):=\{f_t=0\}\subset \ttor$ coexist. Indeed, consider the lattice polyhedron 
\[
\Delta_\nu := \left\lbrace (j_1,j_2,j_3)\in\bR^3 \big| (j_1,j_2)\in \Delta, \, \nu(j_1,j_2)\leqslant j_3 \leqslant \max \nu \right\rbrace
\]
and the corresponding toric $3$-fold $X_{\Delta_\nu}\supset (\bC^*)^3$ with coordinates $(z,w,t)$. The closure of each horizontal section $\{t=constant\}$ in $X_{\Delta_\nu}$ is isomorphic to $X_\Delta$ except for the section $\{t=0\}$ which is a reducible toric surface $\cup_{k\in I} X_{\Delta_k}$, where $\{\Delta_k\}_{k\in I}$ is the subdivision of $\Delta$ induced by $\nu$. 

For a Viro polynomial $f_t$ associated to $\nu$ and a fixed constant $s\in \bC^*$, the curve $\emph{C_s}:=\overline{Z(f_s)}\subset X_\Delta$ can be seen as the compactification of the intersection of the surface $\{f_t=0\}\subset (\bC^*)^3$ with $\{t=s\}$ in $X_{\Delta_\nu}$. The family $C_s\subset X_{\Delta_\nu}$ admits a limit $\emph{C_0}$: this is the reducible curve in $\cup_{k\in I}X_{\Delta_k}$ whose intersection with $X_{\Delta_k}$ is defined by the truncation $f_t^{\Delta_k}$.

\subsection{Logarithmic Gauss map}\label{sec:loggauss}

Let $f$ be a Laurent polynomial with Newton polygon $\Delta$ and denote by $C\subset X_\Delta$ the compactification of $Z(f)\subset \ttor$. Provided that $C$ is smooth, we can define the \emph{logarithmic Gauss map} by
\[
\begin{array}{rcl}
     \gf \quad  : \quad  C &\rightarrow& \bC P^1  \\
      (z,w) & \mapsto & \left[z\cdot\partial_zf(z,w): w\cdot\partial_wf(z,w)\;  \right]
\end{array}
\]
where $(z,w)$ are the coordinates on $\ttor$. Locally, the map $\gamma$ is the composition of any branch of the coordinate-wise logarithm with the usual Gauss map that associates the normal direction to an hypersurface at a smooth point. If a local parametrisation $s\mapsto(z(s),w(s))$ of $C$ is given, the composition of $\gamma$ with the latter is 
\begin{equation}\label{eq:gammaparam}
    s\mapsto \left[ -\frac{d}{ds} \log(w(s)) : \frac{d}{ds} \log(z(s))\right].
\end{equation}

The degree of  $\gf$ is $2\vol(\Delta)$ where $\vol$ is the Euclidean area, provided that $C$ is transverse to every toric divisor. Following \cite{MikhRena}, we refer to the ramification points of $\gf$ as the \emph{Log-inflection points} of $C$.

Let now $f_t$ be a Viro polynomial associated to a $\bZ$-convex function $\nu : \Delta \rightarrow \bR$ and $\{C_t\}_{t\in \bC}$ the corresponding family of curves defined in Section \ref{sec:viropol}. For $ \eps>0$ small enough and any non zero $t$ such that $\vert t \vert <\eps$, the curve $C_t$ is smooth and the corresponding logarithmic Gauss map $\emph{\gamma_t}:=\gamma_{\s{f_t}}$ is well define. The family of maps $\{\gamma_t\}_{\vert t \vert <\eps, t\neq 0}$ extends to a map $\gamma_0$ on $C_0$. To see this, observe that if $\nu=0$ on one of the polygons $\Delta_k$ of the subdivision induced by $\nu$, then the family of maps 
\[
\gamma_t (z,w) \;  = \;  \left[z\cdot\partial_zf_t(z,w): w\cdot\partial_wf_t(z,w)\;  \right]
\]
converges to 
\[
\gamma_0 (z,w) \;  = \;  \left[z\cdot\partial_zf_0(z,w): w\cdot\partial_wf_0(z,w)\;  \right]
\]
which is the logarithmic Gauss map of the curve $\overline{Z(f_0)}\subset X_{\Delta_k}$. By applying a toric change of coordinates of the form $(z,w,t)\mapsto(z,w,tz^aw^b)$, we can make any of the faces $\Delta_k\subset \Delta_\nu$ horizontal. It amounts to replace $\nu$ with $\nu-\ell$ where $\ell$ is the linear function that coincide with $\nu$ on $\Delta_k$. Therefore, we can apply the above reasoning to any element $\Delta_k$ of the subdivision of $\Delta$. This proves the claim. A different viewpoint on the above computation is that the coordinates $(z,w,t)$ induce well defined coordinates $(z,w)$ on the torus $\ttor$ of  each divisors $X_{\Delta_k}\subset X_{\Delta_\nu}$. These coordinates allows us to define compatible logarithmic Gauss maps on each of the irreducible components of $C_0$.

If the Viro polynomial $f_t$ is non-degenerate in the sense of Section \ref{sec:viropol}, then the map $\gamma_0$ has maximal degree $2\vol(\Delta)$ and so does $\gamma_t$ for $\vert t\vert$ small.

\subsection{Log-critical locus}\label{sec:logcrit} Recall that for any lattice polygon $\Delta$, the \emph{moment map} $\mu:X_\Delta \rightarrow \Delta$ is the quotient map of the action of $(S^1)^2$ on $X_\Delta$. After applying a diffeormorphism on $\itr(\Delta)$, the restriction of $\mu$ to the torus $\ttor \subset X_\Delta$ is given by  
\[
\begin{array}{rcl}
     \Log : \ttor &\rightarrow & \bR^2  \\
     (z,w)& \mapsto & (\log\vert z\vert , \log\vert w \vert) 
\end{array}.
\]
For a smooth algebraic curve $C\subset X_\Delta$ given by a Laurent polynomial $f$, the \emph{Log-critical locus} $\emph{\Cr(C)}\subset C$ (denoted alternatively $\emph{\Cr(f)}$) refers to the critical locus of the restriction $\mu: C \rightarrow \Delta$. It was observed in \cite{Mikh} that 
\[
    \Cr(f) = \gf^{-1}(\bR P^1).
\]
It was shown in \cite{L19} that $\Cr(f)$ is smooth for a generic polynomial $f$ within the linear system $\vert \mathcal{L}_\Delta \vert$. We fix once and for all an orientation on $\bR P^1$ so that $\Cr(f)$ inherits an orientation from $\bR P^1$ whenever it is smooth.

Let $f_t$ be a Viro polynomial and let $C_t$, $t\in \bC$, be defined as in the previous section. We define
\[
\emph{\Cr(C_0)}:= \gamma_0^{-1}(\bR P^1).
\]
Equivalently, the locus $\Cr(C_0)$ is the Hausdorff limit of the family of Log-critical loci $\Cr(C_t)$. 

Observe that for any curve $C\subset X_\Delta$, the set $\Cr(C)$ always contains the intersection points of $C$ with the toric divisors of $X_\Delta$, see for instance \cite[Lemma 1.10]{L2}. As a consequence, the set of nodes of $C_0$ is a always a subset of $\Cr(C_0)$.

We say that the Viro polynomial $f_t$ is \emph{Log non-degenerate} if it is non-degenerate  and if the intersection of $\Cr(C_0)$ with any irreducible component of $C_0$ is smooth. In particular, every such piece of $\Cr(C_0)$ inherits an orientation from $\bR P^1$. As observed in Section \ref{sec:loggauss}, the maps $\gamma_t$ have degree $2\vol(\Delta)$ for $\vert t\vert$ small, included for $t=0$. Therefore, the number of connected components of $\Cr(C_t)$ is constant in a 
neighbourhood of $0\in \bC$. 

In a small neighbourhood of any node $p\in C_0$, the topological pair $(C_0, \Cr(C_0))$ is as pictured in Figure \ref{fig:defC_0} (center). For small $t$, the deformation $(C_t, \Cr(C_t))$ of $(C_0, \Cr(C_0))$ is said to be \emph{smooth} if both $C_t$ and $\Cr(C_t))$ are smooth manifolds. At the topological level, there are several possible smooth deformations $(C_t, \Cr(C_t))$ of $(C_0, \Cr(C_0))$ that are compatible with the orientation of $\Cr(C_0)$ and that preserve the number of 
connected components $b_0\big(\Cr(C_0)\big)$. However, there is only one such deformation that connects the Log-critical loci of the two branches of $C_0$ meeting at the node $p$. This deformation is pictured on the left-hand side of Figure \ref{fig:defC_0} and will be referred to as the \emph{connected deformation}. Our interest in this specific deformation will be motivated in Theorem \ref{thm:patchwork}. Similarly, there are several possible deformations that are not smooth. There will be only one such deformation that will be of interest to us. This deformation is pictured on the right-hand side of Figure \ref{fig:defC_0} and will be referred to as the \emph{singular deformation}.

\begin{figure}[h]
    \centering
    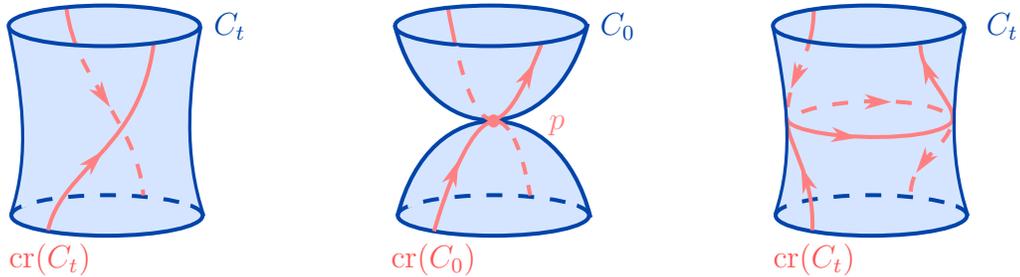
    \caption{The connected deformation (left) and the singular deformation (right) of the pair $(C_0,\Cr(C_0))$ (center)}
    \label{fig:defC_0}
\end{figure}

\subsection{Tropical limit of Viro polynomials}\label{sec:tropgeo}
Let $\nu$ be a $\bZ$-convex function on the lattice polygon $\Delta$ and let $\{\Delta_k\}_{k\in I}$ be the associated subdivision. Let
\[
f(z,w)=\sum_{(i,j)\in\Delta\cap\bZ^2}a_{i,j}z^iw^j,
\]
be a Laurent polynomial and let $f_t$ be the Viro polynomial associated to $f$ and $\nu$, see Section \ref{sec:viropol}. Then the image of $Z(f_t)$ under the map 
\[
\begin{array}{rcl}
     \Log_t : \ttor &\rightarrow & \bR^2  \\
     (z,w)& \mapsto & (\frac{-\log\vert z\vert}{\log(t)} , \frac{-\log\vert w \vert}{\log(t)}) 
\end{array}
\]
admits a limit $\Gamma$ in the Hausdorff sense when $t$ goes to $0$. This limit is usually referred to as the \emph{tropical limit} of $\left\lbrace Z(f_t) \right\rbrace_{t\in\bC}$. The set $\Gamma\subset \bR^2$ is a \emph{tropical curve}, a rectilinear graph which is dual to the subdivision $\{\Delta_k\}_{k\in I}$ of $\Delta$, see for instance \cite{BIMS}. Moreover, any edge $e$ of $\Gamma$ is equipped with a positive integer weight given by the integer length of the edge dual to $e$ in the subdivision $\{\Delta_k\}_{k\in I}$. 

If $p\in\Gamma$, we say that a sequence of points $z_t\in C_t$ \emph{tropically converges} to $p$ (or admits $p$ as tropical limit) if $\lim_{t\to 0} \Log_t(z_t)=p$. It is a classical fact that any sequence $z_t\in C_t$ that converges to a smooth point of $C_0$ tropically converges to a vertex of $\Gamma$. It can be checked by hand, using appropriate coordinate systems.

Eventually, we say that $\Gamma$ is \emph{non-singular} if it is a trivalent graph and if any edge is of weight $1$. By duality, it means that each $\Delta_k$, $k\in I$, is a triangle with area $\frac{1}{2}$.
Again, We refer to \cite{BIMS} for further details.

\subsection{Newton-Puiseux Theorem}\label{sec:npt}

The material of this section will be used exclusively to prove Theorem \ref{thm:asymptotic}. The reader may skip this section at her/his own convenience. Below, we state a simple version of the Newton-Puiseux Theorem for space curves as in \cite[Theorem, Section 3]{Mau}. In order to do so, we need to recall some terminology from the latter reference. 

 Let $v=(v_1,v_2,v_3)\in (\bZ_{\geqslant 0})^3$ be non-zero. A polynomial $g(z,w,t)$ is homogeneous of $v$-order $d$ if $g(s^{v_1},s^{v_2},s^{v_3})$ is a monomial of degree $d$ in the variable $s$. Every polynomial $g(z,w,t)$ can be uniquely written as $g:= \sum_{j\in \bZ_{\geqslant 0}} g_j$ where $g_j$ is homogeneous of $v$-order $j$.
 We refer to $d:=\min\{j\in\bZ_{\geqslant 0}\, \vert\, g_j\neq 0 \}$  as the \emph{initial $v$-order} of $g$ and define the \emph{$v$-initial form} of $g$ by $\emph{\ini_v g}:=g_d$. By extension, we define the $v$-initial form of an ideal $I\in \bC[z,w,t]$ by $\emph{\ini_v I}:=\big(\{\ini_v g \; \vert \; g \in I\}\big)$. A \emph{tropism} of an ideal $I$ is a primitive vector $v\in (\bZ_{\geqslant 0})^3$ such that $\ini_v I$ does not contain any monomial. Geometrically, this is equivalent to require that $-v$ belongs to the dual fan of the Newton polyhedron $\New(g)$ for any $g\in I$. 
 \begin{Example}
 Let $g(z,w,t)=1+z+w+t$. Then $\ini_{(2,1,1)}g=1$ and $\ini_{(0,1,1)}g=1+z$. In fact $(0,-1,-1)$ belongs to the face of the dual fan generated by $(-1,-1,-1)$ and $(1,0,0)$, and $(-2,-1,-1)$ does not belong to the dual fan.
 \end{Example}

The statement below is a simpler version of the Newton-Puiseux Theorem as stated in \cite[Theorem, Section 3]{Mau}.

\begin{thm}\label{thm:npt}
Let $f_1(z,w,t)$ and $f_2(z,w,t)$ be two polynomials such that $X:=\{f_1=f_2=0\}\subset \bC^3$ is one-dimensional. Then, any irreducible component of the reduction of $X$ that passes through $0\in \bC^3$ and that is not contained in a coordinate hyperplane can be parametrised as follows
$$ z=\alpha s^{v_1}+ o(s^{v_1}), \quad  w=\beta s^{v_2}+ o(s^{v_2}) \;\text{ and }\; t=s^{v_3}+ o(s^{v_3})$$
where $(\alpha,\beta)\in \ttor$ and $(v_1,v_2,v_3)$ is a tropism of the ideal $\big(f_1,f_2\big)$.
\end{thm}

\section{Main results}\label{sec:main}

The following theorem is a patchworking theorem for the Log-critical locus of curves $\{C_t\}_{t\in \bC}$ defined by a Viro polynomial $f_t$ as defined in Section \ref{sec:viropol}. Recall that we only consider Viro polynomials $f_t(z,w)$ that are polynomial in $t$, according to Definition \ref{def:div}. 

\begin{theorem}\label{thm:patchwork}
Let $f_t$ be a Log non-degenerate Viro polynomial, see section \ref{sec:logcrit}. Then, there exist $\eps>0$ and a dense open subset $U\subset \{t\in \bC^\ast \big|\, \vert t \vert <\eps\}$ such that for any $t\in U$, the topological pair $\big(C_t, \Cr(C_t)\big)$ is obtained from $\big(C_0,\Cr(C_0)\big)$ by replacing the neighbourhood of every node of $C_0$ with the connected deformation of Figure \ref{fig:defC_0}. In particular, the locus $\Cr(C_t)$ is smooth for any $t\in U$.
\end{theorem}

\begin{Remark}
For $t\notin U$, there is at least one node of $C_0$ whose neighbourhood is replaced by the singular deformation pictured in Figure \ref{fig:defC_0}. In general, it follows from the asymptotic formula \eqref{eq:asymptotic} that the connected and the singular deformations are the only deformations possible for a Log non-degenerate Viro polynomial. If the latter polynomials is real, it follows again from \eqref{eq:asymptotic} that at any given node of $C_0$, one of the deformations appear for $t>0$ and the other deformation appear for $t<0$. 
\end{Remark}

Using the above theorem, we are able to construct projective curves $C$ of any degree such that the Log-critical Locus $\Cr(C)$ is smooth and has a small number of connected components. 
Below, we denote by $\Delta_d$ the standard $d$-simplex, that is $\Delta_d:=\conv\big((0,0), (d,0), (0,d)\big)$.

\begin{theorem}\label{thm:applicationpatchwork}
For any integer $d\geqslant 3$ and any integer $1\leqslant b \leqslant \left(\begin{smallmatrix} d-1\\2\end{smallmatrix}\right)+1$, there exists a smooth Laurent polynomial $f$ with Newton polygon $\Delta_d$ such that the Log-critical locus $\Cr(f)$ is smooth and has exactly $b$ connected components.
\end{theorem}

\begin{Remark}
The above theorem disprove \cite[Conjecture 1]{L19} and points out a missing assumption in \cite[Proposition 6.3]{L19}: using the notations of \cite{L19}, we should assume that $\mathcal{A}$ restricts to an immersion on the Log-critical locus $S(f)$ for the statement to hold.
\end{Remark}

In order to prove Theorem \ref{thm:patchwork}, we need to study the asymptotical behaviour of Log-inflection points along families of curves defined by Viro polynomials. Theorem \ref{thm:patchwork} relies principally on Theorem \ref{thm:asymptotic}, stated and proven in Section \ref{sec:proof}. As a by product of the latter study, we obtain the theorem below which generalises \cite[Theorem 3]{MikhRena}.

\begin{theorem}\label{thm:tropicallimit}
Let $f_t$ be a Log non-degenerate Viro polynomial and let $\Gamma\subset \bR^2$ be the tropical limit of $\{C_t\}_{t\in \bC}$. Then, for any bounded edge $\eps \subset \Gamma$ with multiplicity $m\geqslant 1$ and midpoint $p\in \eps$, there are exactly $2m$ ramification point for $\gamma_t$ in $C_t$ whose tropical limit is $p$. The tropical limit of the remaining ramification points are distributed among the vertices of $\Gamma$.
\end{theorem}

\begin{Remark}
The applications of the asymptotic formulas obtained in Section \ref{sec:proof} reach further than the above statement. Indeed, it allows us to determine the phase-tropical limit of the Log-inflection points. A detailed treatment of this aspect would lead us too far from the main subject here. Let us at least mention the following. Assume that $\Gamma$ is non-singular (and therefore $m=1$ for any bounded edge) and that we consider a continuous path of parameters $t\in \bC$ ending at $0$. Then, we can define the phase-tropical limit $\Gamma_\bC\subset \ttor$ of $C_t$ along this path, see for instance \cite{L} for some background on phase-tropical curves. The phase-tropical limit $\Gamma_\bC\subset \ttor$ is mapped to $\Gamma$ under $\Log$ and the fiber in $\Gamma_\bC$ over any non-vertex point in $\Gamma$ is a geodesic in the argument torus $(S^1)^2$. Then, it follows from the asymptotic formula \eqref{eq:paramsolevenodd} that the two Log-inflection points tropically converging to the midpoint $p$ of a given edge $\eps\in \Gamma$ also converges phase-tropically. Moreover, the two limit points in $\ttor$ are equi-distributed on the geodesic fiber over $p$. Eventually, let us point out that the computation carried in the proof of Theorem \ref{thm:asymptotic} could be extended to non-generic Viro polynomials and lead to a generalisation of Theorem \ref{thm:tropicallimit}, as illustrated in the appendix.
\end{Remark}

\section{Proofs of Theorems \ref{thm:patchwork} and \ref{thm:tropicallimit}}\label{sec:proof}

In this section, we assume that $f_t(z,w)$ is a Log non-degenerate Viro polynomial constructed from a $\bZ$-convex function $\nu:\Delta \rightarrow \bR$, see Sections \ref{sec:viropol} and \ref{sec:logcrit}.
For $U_\eps:=\{t\in \bC \big| \vert t\vert < \eps\}$, we denote by $\mathscr{C}\rightarrow U_\eps$ the family of curve whose fiber over $t\in U_\eps$ is $C_t:=\overline{Z(f_t)}\subset X_{\Delta_\nu}$, see Section \ref{sec:viropol}. 
The parameter $\eps>0$ is assumed to be arbitrarily small. In particular, every curve $C_t$ is smooth for $t\in U_\eps$.

Under the above assumptions, the Log-Gauss map $\gamma_t:=\gamma_{\s{f_t}}:C_t\rightarrow \bC P^1$ has constant degree $2\vol(\Delta)$ for any $t\in U_\eps$, see Section \ref{sec:loggauss}. Moreover, the collection of maps $\gamma_t$, $t\in U_\eps$,  induces a globally defined algebraic map $\Gamma: \mathscr{C}\rightarrow \bC P^1$. It follows from the Riemann-Hurwitz formula that for any node $p\in C_0$, there exists a small neighbourhood $\mathcal{V}\subset \mathscr{C}$ of $p$ which is onto $U_\eps$ and such that $C_t\cap \mathcal{V}$ contains exactly $2$ ramification points of $\gamma_t$ for any $t\neq 0$.  
Indeed, since $f_t$ is Log non-degenerate, the set $\Cr(C_0)$ does not contain any Log-inflection point. Therefore, there exists an arbitrarily small neighbourhood $\mathcal{U}\subset \bC P^1$ of $\gamma_0(p)$ such that the connected component of $\gamma_0^{-1}(\mathcal{U})$ containing $p$ does not contain any Log-inflection point. Define $\mathcal{V}\subset \mathscr{C}$ to be the connected component of $\Gamma^{-1}(\mathcal{U})$ that contains $p$. Now that $\mathcal{V}$ is defined, let us compare the Riemann-Hurwitz formula applied to $\gamma_t: C_t\cap \mathcal{V}\rightarrow \bC P^1$ for $t\neq 0$ to the same formula applied to the pullback of $\gamma_0$ to the normalisation of $C_0\cap \mathcal{V}$.  The  degrees of the maps do not change while the Euler characteristic of the source increases by $2$ at $t=0$ (a cylinder versus the disjoint union of two discs). Therefore, there has to be two distinct Log-inflection points in $C_t\cap \mathcal{V}$ that collide at the node $p$ when $t$ tends to $0$.

Our main goal in this section is to prove the following theorem.

\begin{theorem}\label{thm:asymptotic}
Let $f_t$ be a Log non-degenerate Viro polynomial, see section \ref{sec:logcrit}. For any node $p\in C_0$, there exist $\lambda \in \bC^\ast$ and an integer $\kappa$ such that the image under $\gamma_t$ of each of the two Log-inflection points in $C_t\cap \mathcal{V}$ is equal to
\[
\pm \lambda t^{\kappa/2} + o(t^{\kappa/2})
\] in the appropriate affine chart of $\bC P^1$.
\end{theorem}

In the above statement, the integer $\kappa$ can be either even or odd. In the odd case, the choice of the determination of $t^{\kappa/2}$ is irrelevant.

In order to prove the above theorem, we will compute a parametrisation of the curve described by the position of the Log-inflections points in $C_t\cap \mathcal{V}$ when $t$ varies in $U_\eps$. It follows from the method of Lagrange multipliers that, for a fixed $t$, the Log-inflections points in $C_t$ are exactly the solutions of the system $\{f_t(z,w)=P(z,w,t)=0\}$ where 
\[
P(z,w,t) := \det \left( \begin{array}{cc}
    \partial_z f_t &  \partial_z \gamma_t\\
    \partial_w f_t & \partial_w \gamma_t
\end{array}\right).
\]
Since $\gamma_t$ is a rational function, so is $P$. If we denote by $\emph{N}$ the numerator of $P$, then the system $\{f_t=P=0\}$ is locally  equivalent to the system $\{f_t=N=0\}$ provided that the denominator of $P$ does not vanish. This will be the case in the application below.
We will obtain the parametrisation of the position of the Log-inflection points by applying the Newton-Puiseux Theorem to the system  $\{f_t(z,w)=N(z,w,t)=0\}$. 

For the sake of computation, we first proceed to some changes of coordinates. The node $p\in C_0$ lies on the intersection of two irreducible components of the section $\{t=0\}$ in $X_{\Delta_\nu}$. The latter intersection corresponds to an edge $e=\Delta_j\cap \Delta_k$ where $\Delta_j$ and $\Delta_k$ are polygons of the subdivision of $\Delta$ induced by $\nu$. Using a toric change of coordinates, we can assume without loss of generality that $e$ is directed by $(0,1)$ and that $\nu=0$ on the polygon of the subdivision of $\Delta$ to the right of $e$, say $\Delta_k$. We denote by $\ell$ the integer such that $e$ lies in $\{j_1=\ell\}$ (recall that $(j_1,j_2,j_3)$ are the coordinates on the lattice of monomials $\bZ^3$). The restriction of $\nu$ to $\Delta_j$ is of the form $-\kappa j_1 +\ell\kappa$ where $\kappa \in \bN^*$. 
Furthermore, 
we can multiply $f_t$ by a monomial $z^\alpha w^\beta$ to ensure that $\Delta_\nu$ lies in the positive octant, that is $f_t$ is an honest polynomial in $z,w,t$, and furthermore that $\ell\geqslant 2$.

Applying a toric translation $w\mapsto \alpha w$, $\alpha\neq 0$, if necessary, we can ensure that any sequence of points in $C_t$ that converges to $p\in X_{\Delta_\nu}$ when $t$ tends to $0$ converges to $(0,1,0)$ in the naive partial compactification $\bC^3 \supset (\bC^*)^3$. Therefore, we will work with the coordinates $(z,\tw,t)$ where $\tw=w-1$ on $\bC^3$.

We now define the polynomial $\emph{\tilde{f}_t(z,\tw)}:=f_t(z,\tw+1)=f_t(z,w)$ in the variable $(z,\tw,t)$ and denote by $\wD$ the Newton polyhedron of $\tilde{f}_t$. By a slight abuse of notation, we denote again by $(j_1,j_2,j_3)$ the coordinates on the space of monomials $\bZ^3$ associated to $(z,\tw,t)$.  In the proof of Theorem \ref{thm:asymptotic}, we will use the Newton-Puiseux Theorem on a pair of polynomials, one of which is $\tilde{f}_t$. Therefore, we need to gather some information on the polyhedron $\wD$. Below, we denote 
$$ \emph{H_1}:=\{j_3\geqslant -\kappa j_1 +\ell\kappa\}, \quad \emph{H_2}:=\{j_2\geqslant 0\}, \quad  \emph{H_3}:=\{j_3\geqslant 0\}$$
$$\emph{q_0}:=(\ell,0,0)=\partial H_1\cap \partial H_2\cap \partial H_3, \; \emph{q_1}:=(\ell, 1,0), \; \emph{q_2}:=(\ell+1, 0,0) \; \text{and} \; \emph{q_3}:=(\ell-1, \kappa,0).$$

\begin{Lemma}\label{lem:newton}
The polyhedron $\wD$ has the following properties:
\begin{enumerate}
    \item[(a)]  $\wD$ is contained in each of the half spaces $H_1$, $H_2$ and $H_3$,
    \item[(b)]  $\wD$ contains the lattice point $q_1$ but not the lattice point $q_0$,
    \item[(c)]  $\wD$ contains the lattice points $q_2$ and $q_3$.
\end{enumerate}
In particular, the only facets of $\wD$ whose outer normal vector lies in $(\bR_{<0})^3$ are contained in $\conv(q_0,q_1,q_2,q_3)$ and there are at most two of them.
\end{Lemma}

\begin{proof}
$(a)$ Since $\wD$ is obtained from the Newton polyhedron $\New(f_t)$ after the change of variable $\tw=w-1$, it suffice to show that $\New(f_t)$ is contained in each of the half spaces $H_1$, $H_2$ and $H_3$. This is clear for $H_2$ and $H_3$ since $f_t$ is a polynomial. The facet of $\New(f_t)$ over $\Delta_j$ lies in $H_1$ according to the above choice of coordinates. It follows from convexity that $\New(f_t)$ is contained in $H_1$.

$(b)$ By hypothesis, the truncation of $f_t$ to $e$ has a simple zero at $w=1$. Equivalently, the truncation of $\tilde f_t$ to the line $\{j_1=\ell, j_3=0\}$ has a simple zero at $\tw=0$. In turn, this is equivalent to say that $\wD$ contains $q_1$ but not $q_0$.

$(c)$ Let us first show that there exists a point $q=(\ell+m,0,0)$, $m\geqslant 1$, in $\wD$. Assume towards the contradiction that the intersection of $\wD$ with the first coordinate axis is empty. Thus, the truncation of $f_t$ to the facet $\Delta_k$ is divisible by $(w-1)$, that is the component of $C_0$ lying in $X_{\Delta_k}$ is reducible and contains the curve $\{w=1\}$ as an irreducible component. It contradicts the fact that $f_t$ is non-degenerate. Therefore, the point $q$ exists. Take now $m$ as small as possible. Our aim is to show that $m=1$. To do so, observe that in a neighbourhood of $p$, the curve $C_0\cap X_{\Delta_k}$ can be parametrised in the coordinate $(z,w)$ by 
$$s\mapsto (s+o(s), 1+as^m+o(s^m))$$
with $a\in\bC^*$. By \eqref{eq:gammaparam}, the composition of $\gamma_0$ with the latter parametrisation is
$$ 
\begin{array}{rl}
    s \mapsto \left[ -  \frac{d}{ds} \log(1+as^m+o(s^m)): \frac{d}{ds} \log(s+o(s)) \right] &  =\left[-  \frac{ams^{m-1}+o(s^{m-1})}{1+as^m+o(s^m)}:\frac{1}{s+o(s)}\right] \\ & \\
     & =\left[ -\frac{ams^{m}+o(s^m)}{1+as^m+o(s^m)}:1\right].
\end{array}
$$
The latter map has a critical point at $s=0$ if and only if $m>1$. Equivalently, the point $p$ is a ramification point of the restriction of $\gamma_0$ to $C_0\cap X_{\Delta_k}$ if and only if $m>1$. The point $p$ cannot be such a ramification point since $f_t$ is Log non-degenerate. We conclude that $m=1$ and that $q=q_2$. Since $\Delta_k$ and $\Delta_j$ play a symmetric role, the same arguments apply to $q_3$.
\end{proof}

We now define the polynomial $\emph{\widetilde{N}(z,\tw,t)}:=N(z,\tw+1,t)=N(z,w,t)$ in the variable $(z,\tw,t)$. Since we aim to apply the Newton-Puiseux Theorem to the system $\{\tilde{f}_t(z,\tw)=\widetilde{N}(z,\tw,t)=0\}$, we need now to gather some information on $\wN$.
Below, we denote by $\emph{a}$, $\emph{b}$ and $\emph{c}$ the coefficients such that $az^{\ell+1}$, $bz^\ell\tw$ and $cz^{\ell-1}t^\kappa$ are monomials of $\tilde f_t$. Since those monomials correspond respectively to the lattice point $q_2$, $q_1$ and $q_3$, we know by Lemma \ref{lem:newton} that $abc\in \bC^\ast$. Furthermore, denote $\emph{\delta}:=\min\big\lbrace\kappa/2, \{d\in\bN \vert \frac{\partial^{\ell+d}}{\partial z^{\ell}\partial t^d}\tft(0,0,0)\neq 0\}\big\rbrace$. Then, the polyhedron $\wD$ has exactly two facets whose outer normal vector lies in $(\bR_{<0})^3$ if and only if $\delta<\kappa/2$. If $\delta = \{d\in\bN \vert \frac{\partial^{\ell+d}}{\partial z^{\ell}\partial t^d}\tft(0,0,0)\neq 0\}$, we denote by $\emph{d}$ the coefficient of the monomial $z^\ell t^\delta$ of $\tft$ and declare $d=0$ otherwise. In particular, we have $d=0$ if $\delta=\kappa/2$ and $\kappa$ is odd. We refer to Figure \ref{fig:newton}.

Recall the terminology of Section \ref{sec:npt}. We say that a polynomial $g\in\bC[z,\tw,t]$ is $\emph{o_v(d)}$ if its initial $v$-order is strictly larger than $d$. Plainly, we have the properties
\[
o_v(d_1)+o_v(d_2)= o_v(\min \{d_1,d_2\}) \quad \text{ and } \quad g \cdot o_v(d) = o_v(d+d')
\]
where $d'$ is the initial $v$-order of $g$.

\begin{Lemma}\label{lem:vorder}
Assume that $f_t$ is as in Theorem \ref{thm:asymptotic} and define $v:=(\kappa,2\delta,2)$. Then, there exists a non-zero polynomial $h$ such that  
\begin{equation}\label{eq:zN}
   z\cdot \wN +h\cdot \tft  \;  = \; -(c z^{\ell-1}t^\kappa+a z^{\ell+1}) b^2 z^{2\ell} + o_v(\kappa(3\ell+1)).
\end{equation}
Moreover, the polynomial $\wN$ contains the monomial $\ell^2b^3z^{3\ell-1}\tw$ and contains additionally the monomial $\ell^2b^2 d z^{3\ell-1}t^\delta$ if $d$ is non-zero.
\end{Lemma}
\begin{figure}[t]
    \centering
    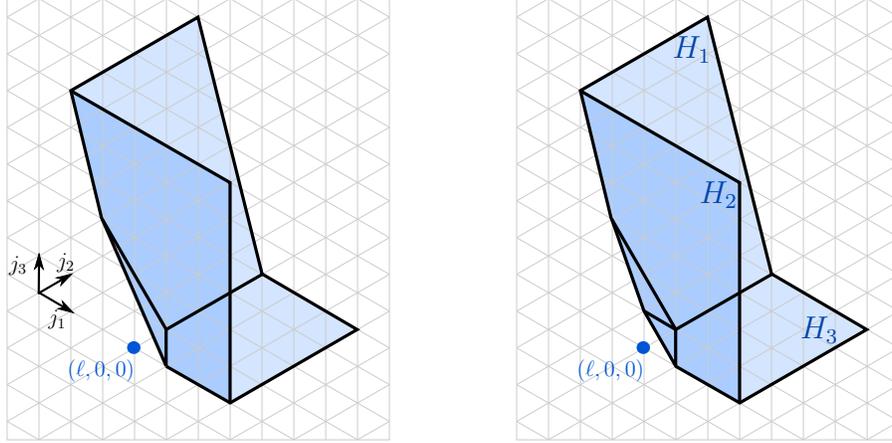
    \caption{The polyhedron $\wD$ for $\kappa=3$ with $\delta=3/2$ (left) and $\delta=1$ (right).}
    \label{fig:newton}
\end{figure}

\begin{proof}
Recall that in the affine chart $\left[u:v\right]\mapsto u/v$ of $\bC P^1$, the map $\gamma_t$ is given by 
\[
\gamma_t(z,w)=\frac{z\cdot \partial_z f_t(z,w)}{w\cdot \partial_w f_t(z,w)}.
\]
Denote $\Gamma_z$ and $\Gamma_w$ the respective numerators of the rational functions $\partial_z \gamma_t$ and $\partial_w \gamma_t$, that is 
$$\Gamma_z:=\partial_z(z\partial_z f_t)\cdot w\partial_w f_t - z\partial_z f_t\cdot \partial_z\big(w\partial_w f_t\big)$$ and 
$$\Gamma_w:=\partial_w(z\partial_z f_t)\cdot w\partial_w f_t - z\partial_z f_t\cdot \partial_w\big(w\partial_w f_t\big).$$
In turn, we have that the numerator $N$ of $P$ is given by 
$$N=\partial_z f_t\cdot \Gamma_w -  \partial_w f_t \cdot \Gamma_z.$$
If we rewrite $\Gamma_z$ and $\Gamma_w$ in the variables $(z,\tw,t)$ as we did for $N$, that is 
$$ \widetilde \Gamma_z(z,\tw,t):= \Gamma_z(z,\tw+1,t) \quad \text{and}\quad  \widetilde \Gamma_w(z,\tw,t):= \Gamma_w(z,\tw+1,t),$$
then we obtain the formulas
$$\widetilde \Gamma_z:=\partial_z(z\partial_z \tilde f_t)\cdot (\tw+1)\partial_{\tw} \tilde f_t - z\partial_z \tilde f_t\cdot \partial_z\big((\tw+1)\partial_{\tw} \tilde f_t\big),$$ 
$$\widetilde\Gamma_w:=\partial_{\tw}(z\partial_z \tilde f_t)\cdot (\tw+1)\partial_{\tw} \tilde f_t - z\partial_z \tilde f_t\cdot \partial_{\tw}\big((\tw+1)\partial_{\tw} \tilde f_t\big)$$
and $$\wN:=\partial_z \tilde f_t\cdot \widetilde \Gamma_w -  \partial_{\tw} \tilde f_t \cdot \widetilde \Gamma_z.$$
Below, we will compute $z\cdot \widetilde N$ modulo $\tft$. To do so, let us first observe that we can write 
\[\tilde f_t=c z^{\ell-1}t^\kappa+Q(\tw,t) z^{\ell}+a z^{\ell+1}+ o_v(\kappa(\ell+1))\]
where $Q\in \bC\left[\tw,t\right]$. This is a consequence of Lemma \ref{lem:newton}. Recall that $\ell\geqslant 2$ by assumption. Then, we have that
\begin{equation} \label{eq:vorder}
\begin{array}{rl}
    z\partial_z \tft & = \;  c (\ell-1) z^{\ell-1}t^\kappa+\ell Q(\tw,t)z^{\ell}+a (\ell+1) z^{\ell+1} + o_v(\kappa(\ell+1)) \\ & \\
     & \equiv \; -c z^{\ell-1}t^\kappa+a z^{\ell+1} + o_v(\kappa(\ell+1)) \mod \tilde f_t, 
\end{array}
\end{equation}
that
\begin{equation} \label{eq:vorder2}
      z\partial_z(z\partial_z \tilde f_t) \; \equiv \; c (1-2\ell)z^{\ell-1}t^\kappa+a (2\ell+1)z^{\ell+1} + o_v(\kappa(\ell+1)) \mod \tilde f_t,
\end{equation}
\begin{equation} \label{eq:vorder3}
      \partial_{\tw}( z\partial_{z}\tft )\; = \;  \ell \partial_{\tw}\tft + o_v(\kappa(\ell+1)-2\delta)
\end{equation}
and eventually that
\begin{equation} \label{eq:vorder4}
      z\partial_z( (\tw+1)\partial_{\tw}\tilde f_t )\; = \;  \ell (\tw+1)\partial_{\tw}\tilde f_t + o_v(\kappa(\ell+1)-2\delta).
\end{equation}
In order to compute $z\cdot \wN$, we compute separately $z\partial_z \tft\cdot \widetilde \Gamma_w$ and   $z \partial_{\tw} \tft \cdot \widetilde \Gamma_z$. Using \eqref{eq:vorder} and \eqref{eq:vorder3}, we obtain
\[
\begin{array}{rl}
    z\partial_z \tft\cdot \widetilde \Gamma_w & = \;  {\color{BlueViolet}z\partial_z \tft} \big( {\color{Emerald}\partial_{\tw}(z\partial_z \tft)}\cdot (\tw+1)\partial_{\tw} \tilde f_t - {\color{Peach}z\partial_z \tft \cdot \partial_{\tw}\big((\tw+1)\partial_{\tw} \tft\big)} \big)\\
     & \\
     & \equiv \; {\color{BlueViolet}\big(-c z^{\ell-1}t^\kappa+a z^{\ell+1} + o_v(\kappa(\ell+1))\big)}\big({\color{Emerald}(\ell \partial_{\tw}\tft + o_v(\kappa(\ell+1)-2\delta))}(\tw+1)\partial_{\tw} \tft \\ & \\ & \quad - {\color{Peach}\big(-c z^{\ell-1}t^\kappa+a z^{\ell+1} + o_v(\kappa(\ell+1))\big) \partial_{\tw}\big((\tw+1)\partial_{\tw} \tft\big)}
     \big) \\ & \\
     & \equiv \; (-c z^{\ell-1}t^\kappa+a z^{\ell+1}) \ell (\tw+1)(\partial_{\tw} \tft)^2 + o_v(\kappa(3\ell+1)) \mod \tft.
\end{array}
\]
Observe that the product of the orange term with ${\color{BlueViolet}z\partial_z \tft}$ has initial $v$-order $\kappa(3\ell+2)$ at least. Hence, it is $o_v(\kappa(3\ell+1))$ and does not contribute to the last formula obtained above. Using \eqref{eq:vorder}, \eqref{eq:vorder2} and \eqref{eq:vorder4}, we obtain
\[
\begin{array}{rl}
    z \partial_{\tw} \tft \cdot \widetilde \Gamma_z & = \;  \partial_{\tw} \tft \big({\color{BlueViolet}z\partial_z(z\partial_z \tft)}\cdot (\tw+1)\partial_{\tw} \tft - {\color{Emerald}z\partial_z \tft}\cdot {\color{Peach}z\partial_z\big((\tw+1)\partial_{\tw} \tft\big)}\big) \\ & \\
     & \equiv \; \partial_{\tw} \tft \big({\color{BlueViolet}(c (1-2\ell)z^{\ell-1}t^\kappa+a (2\ell+1)z^{\ell+1} + o_v(\kappa(\ell+1)))}\cdot (\tw+1)\partial_{\tw} \tft \\ & \\ & \quad
      - {\color{Emerald}(-c z^{\ell-1}t^\kappa+a z^{\ell+1} + o_v(\kappa(\ell+1)))}\cdot {\color{Peach}(\ell (\tw+1)\partial_{\tw}\tft + o_v(\kappa(\ell+1)-2\delta))}\big)\\ & \\
     & \equiv \; (c(1-\ell)z^{\ell-1}t^\kappa+a(\ell+1) z^{\ell+1}) \ell (\tw+1)(\partial_{\tw} \tft)^2 + o_v(\kappa(3\ell+1)) \mod \tft.
\end{array}
\]
Eventually, since $\partial_{\tw} \tft=bz^{\ell}+o_v(\kappa\ell)=(\tw+1)(\partial_{\tw} \tft)$, we conclude that 
\[
\begin{array}{rl}
    z\cdot \wN  &  \equiv \; -(c z^{\ell-1}t^\kappa+a z^{\ell+1}) (\tw+1)(\partial_{\tw} \tft)^2 + o_v(\kappa(3\ell+1)) \mod \tft \\ & \\
     & \equiv \; -(c z^{\ell-1}t^\kappa+a z^{\ell+1}) b^2 z^{2\ell} + o_v(\kappa(3\ell+1)) \mod \tft.
\end{array}
\]

For the second part of the statement, we compute $\wN$ up to $o_{v'}(\kappa(3\ell-1)+1)$ with $v':=(\kappa,1,2)$. Since $2\delta>1$, we have that $\partial_z \tft = \ell b z^{\ell-1}\tw+ o_{v'}(\kappa(\ell-1)+1)$ and $\partial_{\tw} \tft = b z^{\ell-1}+ o_{v'}(\kappa(\ell-1))$. In turn, we obtain that
\[
\begin{array}{rl}
    \wN & = \; \ell b z^{\ell-1}\tw \big((\ell b z^\ell)(b z^\ell) - {\color{Emerald}(\ell b z^{\ell}\tw)\partial_{\tw}((\tw+1)\partial_{\tw}\tft)}\big)\\ & \\
     & \quad - b z^\ell {\color{Peach}\big((\ell^2 b z^{\ell-1}\tw)(b z^\ell)-(\ell b z^\ell \tw)(\ell b z^{\ell-1})\big)} + o_{v'}(\kappa(3\ell-1)+1)\\ & \\
     & = \ell^2 b^3 z^{3\ell-1}\tw+ o_{v'}(\kappa(3\ell-1)+1)\; 
\end{array}
\]
where the last equality comes from the fact that the terms in orange cancel each other and that the final contribution of the green term is $o_{v'}(\kappa(3\ell-1)+1)$. This implies that $\wN$ contains the monomial $\ell^2b^3z^{3\ell-1}\tw$ and in turn that $z^{-1}\cdot h$ contains the monomial $\ell^2 b^2 z^{2\ell-1}$. In particular, the polynomial $h$ is non-zero. The multiplication of $\ell^2 b^2 z^{2\ell-1}$ with the monomial $d z^\ell t^\delta$ of $\tft$ leads to the monomial $\ell^2 b^2 z^{3\ell-1}t^\delta$ of $\wN$.
\end{proof}

\begin{Corollary}\label{cor:newton}
The polyhedron $\wN$ is contained in the translation by $(2\ell-1,0,0)$ of each of the half spaces $H_1$, $H_2$ and $H_3$. The only facets of $\wN$ whose outer normal vector lies in $(\bR_{<0})^3$ are obtained by translation by $(2\ell-1,0,0)$ from the facets of $\wD$ having the same property.
\end{Corollary}

\begin{proof}
This follows from the Lemmas \ref{lem:newton} and \ref{lem:vorder} and the fact that $z^{-1}\cdot h$ contains the monomial $\ell^2 b^2 z^{2\ell-1}$.
\end{proof}

\begin{proof}[Proof of Theorem \ref{thm:asymptotic}] In order to prove the theorem, we will find a parametrisation of the branches of the curve $\mathcal{C}:=\{\tilde f_t=\widetilde N=0\}$ that pass through $0\in \bC^3$.

Let us first consider the case when $\delta< \kappa/2$, that is, there exists an extra vertex $q_4$ of $\wD$ in $\conv(q_0,q_2,q_3)$. According to Lemma \ref{lem:newton}, the polyhedron $\wD$ has exactly two facets whose outer normal vector lies in $(\bR_{<0})^3$, namely $F_1:=\conv(q_1,q_2,q_4)$ and $F_2:=\conv(q_1,q_3,q_4)$ with respective normal vectors $-(\delta,\delta,1)$ and $-(\kappa-\delta,\delta,1)$. By Corollary \ref{cor:newton}, the polyhedron $\New(\widetilde N)$ contains the facets $F_1':=F_1+(2\ell-1,0,0)$ and $F_2':=F_2+(2\ell-1,0,0)$. Consider now the partial toric compactification  $X_D \supset (\bC^\ast)^3$ given by the non-compact polyhedron $D$ that is the positive cone over $q_0$ of the facets $F_1$ and $F_2$. The latter are the only compact facets of $D$ and the remaining 3 facets are supported on $H_1$, $H_2$ and $H_3$. Therefore, the branches of $\mathcal{C}$ that pass through $0\in \bC^3$ are exactly the branches whose closure in $X_D$ intersects the union of toric divisors  $X_{F_1}\cup X_{F_2}\subset X_D$.

The intersection of $\mathcal{C}$ with $X_{F_1}$ (respectively $X_{F_2}$) is given by $S_1:=\{\tilde f_t^{F_1}=\widetilde N^{F_1'}\}$ (respectively $S_2:=\{\tilde f_t^{F_2}=\widetilde N^{F_2'}\}$). By the Bernstein-Kushnirenko Theorem, both $S_1$ and $S_2$ consists of a single point. By Lemma 
\ref{lem:vorder}, we have that $\tilde f_t^{F_1\cap F_2}=b^2\ell^2 z^{2\ell-1} \cdot \widetilde N^{F_1'\cap F_2'}$. Therefore, we have that $S_1=S_2:=S$ and that $S \in X_{F_1}\cap X_{F_2}$. 

By the Newton-Puiseux Theorem \cite[Theorem, Section 3]{Mau}, we can parametrise each branch of $\mathcal{C}$ passing through $S$ by an auxiliary parameter $s$ as follows 
\begin{equation}\label{eq:paramsolevenodd}
 s\mapsto \big(\alpha s^{v_1}+ o(s^{v_1}),\beta s^{v_2}+ o(s^{v_2}), s^{v_3}\big)
\end{equation}
where $\alpha, \beta \in \bC^*$ and $(v_1,v_2,v_3)\in \bN^3$ is primitive. We claim that there are exactly two such branches with $(v_1,v_2,v_3)=(\kappa/2,\delta,1)$ if $\kappa$ is even and that there is exactly one such branch with $(v_1,v_2,v_3)=(\kappa,2\delta,2)$ if $\kappa$ is odd.
To see this, let us perturb slightly $\mathcal{C}$ into a curve $\mathcal{C}'$ such that the corresponding points $S_1'$ and $S_2'$ becomes distinct points living in the respective tori $\ttor$ of $X_{F_1}$ and $X_{F_2}$. In this case, each of $S_1'$ and $S_2'$ corresponds to a single branch of $\mathcal{C}'$ parametrised respectively by 
$$
    \big(\alpha_1 s^{\delta}+ o(s^{\delta}),\beta_1 s^{\delta}+ o(s^{\delta}), s\big) \quad \text{and} \quad \big(\alpha_2 s^{\kappa-\delta}+ o(s^{\kappa-\delta}),\beta_2 s^{\delta}+ o(s^{\delta}), s\big)
$$
where the vectors of exponents are the inner normal to the corresponding facets, see for instance \cite[Section 2]{Till}. Thus, the sum of vector of exponents of over all the  branches of $\mathcal C$ passing through $S$ should be equal to the sum of the above vectors of exponents, namely $(\kappa,2\delta,2)$. Indeed, each coordinate of the sum is the total intersection number of $\mathcal{C}$ with a given coordinate hyperplane in a neighbourhood of $0\in \bC^3$ and is therefore invariant by small perturbation of $\mathcal{C}$. Since the vector of exponent of a single branch of $\mathcal{C}$ has to be primitive, there is a single branch with exponent vector $(\kappa,2\delta,2)$ if $\kappa$ is odd and two branches with exponent vector $(\kappa/2,\delta,1)$ otherwise. The claim follows.

Whether $\kappa$ is odd or even, the composition of the parametrisation \eqref{eq:paramsolevenodd} with $\tft$ has to be identically zero as a power series in the variable $s$. This applies to the coefficient of the monomial with lowest degree, coming from $bz^\ell\tw+d t^\delta$, that is $\alpha^\ell(b\beta+1)$. We deduce that $\beta=-1/d$. Applying the same reasoning to the polynomial $z\cdot \wN+ h\cdot \tft$ of Lemma \ref{lem:vorder}, we deduce that
\begin{equation}\label{eq:alpha}
    c+a\alpha^2=0.
\end{equation}
We now compute the composition of $\gamma_t(z,w)=\gamma_t(z,\tw+1))=\frac{z\partial_z \tilde f_t(z,\tw)}{(\tw+1)\partial_{\tw}\tilde f_t(z,\tw)}$ with the parametrisation \eqref{eq:paramsolevenodd}. Below, we use \eqref{eq:vorder} in the first equality together with the fact that the composition of \eqref{eq:paramsolevenodd} with $\tft$ is identically $0$, then we use \eqref{eq:alpha} in the third equality to obtain
\begin{equation}\label{eq:asymptotic}
\begin{array}{rl}
    \gamma_t(z(s), \tw(s)+1)) & = \displaystyle \frac{c(\ell-1)(\alpha s^{v_1})^{\ell-1}(s^{v_3})^\kappa+a(\ell+1)(\alpha s^{v_1})^{\ell+1}+o(s^{(\ell+1)v_1})}
    {(1+\beta s^{v_2}+o(s^{v_2}))(b(\alpha s^{v_1})^\ell+o(s^{\ell v_1}))} \\
    & \\
     & = \displaystyle s^{v_1} \frac{c(\ell-1)\alpha^{-1}+a(\ell+1) \alpha}{b}+ o(s^{v_1}) \; = \; s^{v_1} \frac{2a\alpha}{b}+ o(s^{v_1})\\
    & \\
    & = \; s^{v_1} \frac{2\sqrt{ac}}{b}+ o(s^{v_1}) \; = \; t^{\kappa/2}\,  \frac{2\sqrt{ac}}{b}+ o(t^{\kappa/2}).
\end{array}
\end{equation}
When $\kappa$ is even, the two determinations of $\sqrt{ac}$ gives the two branches. When $\kappa$ is odd, the two determinations of $\sqrt{ac}$ are compensated by the determinations of $t^{\kappa/2}$ and gives only one branch, as expected. The result follows.

It remains to prove the statement when $\delta=\kappa/2$. In that case, it follows from Corollary \ref{cor:newton} that $F:=\conv(q_1,q_2,q_3)$ is the only facet of $\wD$ whose outer normal vector lies in $(\bR_{<0})^3$ and that $F+(2\ell-1,0,0)$ is a facet of $\New(\wN)$. As above, we can consider the partial toric compactification of $(\bC^\ast)^3$ given by the cone $D$ of $F$ over $q_0$. It follows now from \eqref{eq:zN} that the system $\{\tft^{F}=\wN^{F}\}$ has two distinct solutions inside the torus $\ttor\subset X_D$. We deduce as before that the branches of $\mathcal{C}$ passing through these solution are parametrised as in \eqref{eq:paramsolevenodd} with $(v_1,v_2,v_3)=(\kappa/2,\delta,1)$ if $\kappa$ is even and $(v_1,v_2,v_3)=(\kappa,2\delta,2)$ if $\kappa$ is odd. The rest of the proof is identical to the one of the previous case.
\end{proof}

As in Section \ref{sec:viropol}, we denote by $\vert \mathcal{L}_\Delta \vert$ the linear system of curves given by a Laurent polynomial with Newton polygon included in $\Delta$. Before we prove Theorem \ref{thm:patchwork}, let us recall that the space $\mathcal{D}\subset \vert \mathcal{L}_\Delta \vert$ of Laurent polynomial $f$ such that $\Cr(f)$ is singular is a semi-algebraic set of real codimension 1, see \cite[Theorem 1]{L19}. Recall also that the locus $\mathcal{D}$ corresponds to those curves having at least one Log-inflection point $p$ such that $\gamma(p)\in \bR P^1$, see \cite[Proposition 1.1]{L19}.

\begin{proof}[Proof of Theorem \ref{thm:patchwork}]
In order to prove the theorem, we will study the pullback to $U_\eps$ of $\mathcal{D}$ under the map $F:U_\eps \rightarrow \vert \mathcal{L}_\Delta \vert$ sending $t$ to $f_t$.

For any node $p\in C_0$, denote by $\ell\ell: U_\eps\rightarrow \Sym_2(\bC P^1)$ the map that associates to any $t$ the image under $\gamma_t$ of the two Log-inflection points in $C_t\cap \mathcal{V}$. In turn, define $\mathcal{D}_p:= \ell\ell^{-1}(\Sym_2(\bR P^1))\subset U_\eps$.

We claim that $F^{-1}(\mathcal{D})=\cup_p \mathcal{D}_p$ where the union runs over all the nodes $p\in C_0$.  To see this, observe first that we have the obvious inclusion $F^{-1}(\mathcal{D})\supset\cup_p \mathcal{D}_p$. The fact that the latter inclusion is an equality relies on the assumption that $f_t$ is Log non-degenerate. Indeed, all the ramification points of $\gamma_t$ that do not converge to a node of $C_0$ are away from $\gamma_t^{-1}(\bR P^1)$ for all $t\in U_\eps$. Therefore, the only branches of $F^{-1}(\mathcal{D})$ come from $\mathcal{D}_p$ for some node $p\in C_0$.

By Theorem \ref{thm:asymptotic}, the set $\mathcal{D}_p$ is diffeomorphic to $\{t\in U_\eps \big| \, \arg(t^\kappa)=1\}$, for $\eps$ small enough. This implies that the complement of $F^{-1}(\mathcal{D})=\cup_p \mathcal{D}_p$ is an open dense subset.

It remains to show that for any $t\in U$, the topological pair $\big(C_t\cap \mathcal{V}, \Cr(C_t)\cap \mathcal{V}\big)$ is the connected deformation of $\big(C_0\cap \mathcal{V}, \Cr(C_0)\cap \mathcal{V}\big)$. To see this, observe that for a given node $p\in C_0$,  the asymptotic formula of Theorem \ref{thm:asymptotic} implies that 
for any $t\in U_\eps\setminus \mathcal{D}_p$, the two branching points $\ell\ell(t)$ are on distinct hemispheres of $\bC P^1\setminus \bR P^1$. Therefore, any small simply connected neighbourhood of $p$ deforms into a cylinder in $C_t$ that is mapped under $\gamma_t$ to a disc in $\bC P^1$. This disc is cut into two pieces by $\bR P^1$, each piece containing a simple branching point. The latter cylinder together with the pullback of $\bR P^1$ is therefore the connected deformation pictured in Figure \ref{fig:defC_0}. The result follows.
\end{proof}

\begin{proof}[Proof of Theorem \ref{thm:tropicallimit}]
Let $p\in C_0$, $e=\Delta_k\cap\Delta_j$, $U_\eps$ and $\mathcal{V}$ be as above. We still work with the coordinate system introduced at the beginning of this section. In particular, the edge $e$ is directed by $(0,1)$. Then, the extremities of the bounded edge $e^\star\subset \Gamma$ dual to $e$ have coordinates $(0,0)$ and $(0,-\kappa)$. This is a classical fact from tropical geometry: it follows from the duality between $\Gamma$ and the subdivision of $\Delta$ and from \cite[Theorem 2.12]{BIMS}.
According to the proof of Theorem \ref{thm:asymptotic}, the two ramification points in $C_t\cap \mathcal{V}$ have coordinates
\[
\big( \sqrt{-ac}\cdot t^{\kappa/2} +o(t^{\kappa/2}), 1+o(1), t \big)
\]
in $(\bC^*)^3$. It follows that the tropical limit of any of these two points is $(-\kappa/2,0)\in \bR^2$, that is the midpoint of $e^\star$. Since $f_t$ is non-degenerate, the multiplicity $m$ of $e^\star$ coincides with the number of nodes of $C_0$ on the toric divisor corresponding to $e$, by duality between $\Gamma$ and $\{ \Delta_k\}_{k\in I}$. Thus, there are $2m$ ramification points tropically converging to the midpoint of $e^\star$. To see that there are exactly $2m$ ramification points tropically converging to the midpoint of $e^\star$, observe that all ramification points of $\gamma_t$ converge either to a node of $C_0$ or to a ramification point of $\gamma_0$. Eventually, recall that any family of point $p_t\in C_t$ that converge to a point in $C_0$ that is not a node tropically converges to a vertex of $\Gamma$. 

To conclude the proof, it remains to see that the above computation does not depend on the choice of coordinates made at the beginning of this section. This follows from the following observations. First, we applied a toric change of coordinates on $(z,w)$ to make the edge $e$ vertical in the $(j_1,j_2)$-plane. This results in an integer affine transformation of the tropical plane $\bR^2$. Then, we applied a transformation of the form $(z,w,t)\mapsto (t^\beta z,w,t)$ to make the face of $\Delta_\nu$ over $\Delta_k$ horizontal in the $(j_1,j_2,j_3)$-space. This results in a translation along the first coordinate axis in the tropical plane. It is now clear that the conclusion of the above computations is not affected by the latter changes of coordinates. The result follows. 
\end{proof}

\section{Curves with prescribed Log-critical locus}\label{sec:auxilliarycurves}

We aim to describe the $\Log$-critical locus of curves defined by polynomials of the form 
\begin{equation}\label{eq:f}
    f(z,w):=
    w\cdot \prod_{j=1}^d (z-b_j) + \prod_{j=1}^{d+1} (z-a_j)=:
    w\cdot q(z)+p(z)
\end{equation}
for specific choices of the parameters $a_1, \cdots ,a_{d+1}, b_1, \cdots, b_d$. 
The corresponding curve $Z(f)\subset(\bC^*)^2$ can be naturally compactified in the Hirzebruch surface $\Sigma_1$. The projection onto the first coordinate gives the rational parametrization $\varrho : z \mapsto \big(z,-\frac{p(z)}{q(z)}\big)$ of $Z(f)$.
Recall that the $\Log$-critical locus $\Cr(f)$ is given by $\Cr(f)=\gf^{-1}(\bR P^1)$ where $\gf$ is the logarithmic Gauss map 
\[
\gf(z,w)\;=\; \left[z \cdot \partial_z f(z,w)\, : \; w \cdot \partial_w f(z,w) \right] \;=\; \left[z \cdot (p'(z)+w\cdot q'(z))\, : \; w\cdot q(z) \right].
\]
Denote by $\tgf$ the composition of $\gf$ with the parametrization $\varrho$. In the affine chart $\left[u:v\right]\mapsto u/v$ of $\bC P^1$,
the map $\tgf$ is given by
\begin{equation}\label{eq:gammatilde}
    \tgf(z)\;=\; z \cdot \Big(\dfrac{q'(z)}{q(z)}- \dfrac{p'(z)}{p(z)}\Big)\;=\; \frac{-z}{z-a_{d+1}}+\sum_{j=1}^d \frac{z}{z-b_{j}} -\frac{z}{z-a_{j}}
\end{equation}
\[\hspace*{6,5cm}=\; -1- \frac{a_{d+1}}{z-a_{d+1}}+\sum_{j=1}^d \frac{b_j}{z-b_{j}} -\frac{a_j}{z-a_{j}}.\]
By definition, the pullback of $\Cr(f)$ under $\varrho$ is the set $\tgf^{-1}(\bR P^1)$. In particular, the latter set contains the points $a_1,\cdots,a_{d+1}, b_1, \cdots, b_d, 0$ and $\infty$ since
\[\tgf(\{a_1,\cdots,a_{d+1},b_1,\cdots,b_{d}\})=\{\infty\}, \quad \tgf(0)=0 \quad \text{and} \quad \tgf(\infty)=-1.\]

From now on, we take the coefficients $a_1, \cdots ,a_{d+1}, b_1, \cdots, b_d$ to be real and satisfying 
\begin{equation}\label{eq:ajbj}
    b_1\, < \, a_1 \, <\,  b_2 \, < \, a_2 \, <  \cdots \, < \,  b_{d} \, < \, a_{d} \, < \, a_{d+1} \, < \, 0.
\end{equation}

Below, we describe $\tgf^{-1}(\bR P^1)$ (see Proposition \ref{prop:critloc}) and study how the latter set is affected by the change of variable $z\mapsto z+\lambda i$ (see Proposition \ref{prop:deformation}).
For convenience, let us introduce the following notations for any $1\leqslant j \leqslant d$
\[
\ell_j(z):=\frac{z}{z-b_{j}} -\frac{z}{z-a_{j}}, \; \ell(z):=\frac{-z}{z-a_{d+1}}, \; c_j:=\frac{a_j+b_j}{2}, \;  \eps_j= a_j-b_j \; \text{and} \; \eps:= \max_{1\leqslant j \leqslant k} \eps_j.
\]

\begin{Proposition}\label{prop:eps}
The restriction of $\tgf$ to $\bR P^1$ has $2d+1$ zeroes and $2d$ simple critical points, provided that $\eps$ is small enough. Moreover, the value of $\tgf$ at each real critical point is strictly negative.
\end{Proposition}

The graph of $\tgf$ is represented in Figure \ref{fig:graphgamma}, for $d=3$.

\begin{proof}
Let us first show that $\tgf$ has $2d+1$ real zeroes. Since the limit of $\tgf$ on both extremities of each interval $(a_j,b_{j+1})$ is $+\infty$ for $1\leqslant j < d$, the inequality $\tgf\big(\frac{a_j+b_{j+1}}{2}\big)<0$ implies the existence of $2$ real zeroes on each such interval. To ensure that these inequalities hold, observe that 
\[
\ell_j(z) \; = \; \frac{b_j}{z-b_{j}} -\frac{a_j}{z-a_{j}} \; = \; \frac{(b_j-a_j)z}{(z-b_{j})(z-a_{j})} \; = \; \frac{-\eps_j \cdot z}{(z-b_{j})(z-a_{j})}
\]
and consequently that for any $1\leqslant j,q \leqslant d$, we have 
\[
\lim_{\eps_q\rightarrow 0} \ell_q\big(\frac{a_j+b_{j+1}}{2}\big)=0.
\]
For $1\leqslant j < d$, it implies that 
\[
\lim_{\eps \rightarrow 0} \tgf \big(\frac{a_j+b_{j+1}}{2}\big)= \ell\big(\frac{a_j+b_{j+1}}{2}\big) <0.
\]
and in turn that the value $\tgf\big(\frac{a_j+b_{j+1}}{2}\big)$ satisfies the sought inequality, for $\eps$ small enough.

This provides us with $2(d-1)$ many real zeroes. Each of the intervals $(-\infty,b_1)$, $(a_d,a_{d+1})$ contains a zero since 
\[
\lim_{z\rightarrow -\infty} \tgf(z)=-1, \; \lim_{z\rightarrow b_1^-} \tgf(z)=+\infty, \; \lim_{z\rightarrow a_d^+} \tgf(z)=+\infty \text{ and } \lim_{z\rightarrow a_{d+1}^-} \tgf(z)=-\infty.
\]
The remaining zero is at $0$. In total, it provides us with $2d+1$ real zeroes.

Let us now prove that $\tgf$ has $2d$ real simple critical points. First, we claim that the restriction of $\tgf$ to each of the $2d$ intervals of 
\[
(-\infty,a_d) \setminus \{b_1,a_1,b_2,a_2,\cdots,b_{d-1}, a_{d-1}, b_d\}
\]
has a global extreme value which is strictly negative. For $(-\infty,b_1)$, the claim follows from the fact that 
\[
\lim_{z\rightarrow -\infty} \tgf(z)=1^- \quad \text{and} \quad \lim_{z\rightarrow b_1^-} \tgf(z)=+\infty.
\]
For each remaining interval $(s,t)$, the claim follows from the fact that 
\[
\lim_{z\rightarrow s^+} \tgf(z) = \lim_{z\rightarrow t^-} \tgf(z) \in \{-\infty, +\infty\}
\]
and the location of the real zeroes determined above.

It remains to show that each such extreme point is simple and that $\tgf$ has no other real critical point, for $\eps$ small enough. To see this, recall that by Riemann-Hurwitz formula, exactly $4$ critical points of $\tgf$ collapse at $c_j$ when $\eps_j$ tends to $0$, for $1\leqslant j \leqslant k$. Indeed, the degree of $\tgf$ drops by $2$ since exactly two poles disappear, namely $a_j$ and $b_j$, and the degree of the ramification divisor of $\tgf$ drops therefore by $4$. We claim that for $\eps$ small enough, exactly two of these ramification points are among the real critical points of $\tgf$ exhibited above and the remaining two are complex conjugated. Moreover, there is no other real critical point since the function $\tgf$ is $1$-to-$1$ when $\eps=0$.

In order to prove the above claim, it suffices to let all the $\eps_q$, with $q\neq j$, tend to $0$ and study the distribution of the $4$ remaining critical points. In other words, it amounts to consider the case $d=1$, in particular $\eps_1=\eps$. Then, we have
\[
\begin{array}{rcl}
     \tgf' (z) & = &\displaystyle \frac{(a_1-b_1)(z^2-a_1b_1)}{(z-a_1)^2(z-b_1)^2}  +\ell'(z)\\
     & & \\
     & =  &\displaystyle \frac{\eps(h^2+2c_1h-\eps^2/4)+(h^2-\eps^2/4)^2\cdot \ell'(z)}{(h^2-\eps^2/4)^2}\\ 
\end{array}
\]
where we substituted $z=h+c_1$. In order to study the behaviour of the $4$ ramification points of $\tgf$ converging to $c_1$ when $\eps$ tends to $0$, we need to analyse the singularity of the numerator of $\tgf$ in the variables $(h,\eps)$ at $(0,0)$. Since $\alpha:=\ell'(c_1)$ is a strictly negative number, we deduce that the latter singularity has Newton diagram with vertices $(4,0)$, $(1,1)$ and $(0,3)$ and corresponding truncation $\alpha h^4-2c_1h\eps-\eps^3/4$. By the Newton-Puiseux theorem, we deduce that the singularity consists of two real branches parametrised respectively by $\eps\mapsto \big(-\frac{1}{8c_1}\eps^2 +o(\eps^2),\eps\big)$ and $h\mapsto\big(h, \frac{\alpha}{2c_1} h^3+o(h^3)\big)$.
For $t>0$ and arbitrarily small, the line $\eps=t$ intersects the first branch in one real point and the second branch in one real point and two complex conjugated points. The result follows.
\end{proof}

\begin{figure}[h]
    \centering
    \includegraphics[scale=0.5]{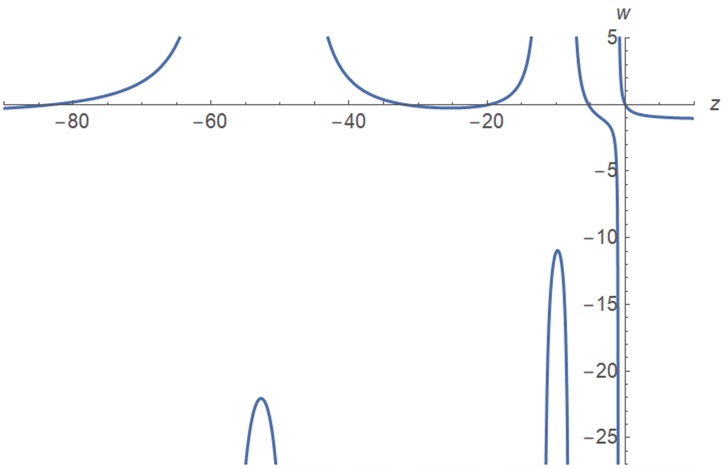}
    \caption{The graph of $\protect\tgf$ for $d=2$, $b_1=-58$, $a_1=-48$, $b_2=-12$, $a_2=-8$ and $a_3=-1$.}
    \label{fig:graphgamma}
\end{figure}

\begin{Remark}\label{rem:critical}
As we have seen in the above proof, there are exactly two real critical points of $\tgf$ in an arbitrarily small neighbourhood of each point $c_j$ for $1\leqslant j \leqslant d$. Moreover, those two critical points, let us say $p_j$ and $q_j$, are such that $p_j<b_j<q_j<a_j$, see Figure \ref{fig:graphgamma}.
\end{Remark}

\begin{Proposition}\label{prop:critloc}
The preimage of $\bR P^1$ under $\tgf$ is the union of $\bR P^1$ with $d$ pairwise disjoint smoothly embedded circles in $\bC P^1$,  provided that $\eps$ is small enough. Each such circle intersects $\bR P^1$ at the $2$ real critical points $p_j$ and $q_j$ of $\tgf$ near $c_j$, for some $1\leqslant j \leqslant d$.
\end{Proposition}

\begin{Lemma}\label{lem:preimage}
Let $g:\Sigma\rightarrow \bC P^1$ be a real meromorphic function on a compact Riemann surface $\Sigma$. Then, the preimage $g^{-1}(\bR P^1)\subset \Sigma$ is a union of smoothly immersed circles such that each circle (respectively $2$ circles) self-intersects (respectively intersect) at worst transversely. Moreover, the set  $g^{-1}(\bR P^1)$ cannot contain a smooth open arc $\alpha$ such that $\overline{\alpha}\setminus \alpha$ is a single point which is critical for $g$. 
\end{Lemma}

\begin{proof}
Fix an arbitrarily chosen Riemannian metric on $\Sigma$ and denote by $UT\, \Sigma$ the unitary tangent bundle with respect to the latter metric. Fix on orientation on $\bR P^1$. Then, the $1$-fold $g^{-1}(\bR P^1)\subset \Sigma$ inherits an orientation on its smooth locus and therefore lifts to $UT\, \Sigma$. Since the local model of $g^{-1}(\bR P^1)$ at singular points is $\{ \vert z \vert < 1 \, , \; z^n\in \bR \}$, the closure $\mathscr{C}$ of this lift is a smooth compact oriented $1$-fold covering $\bR P^1$. Thus, the set $\mathscr{C}$ is a disjoint union of smoothly embedded circles in $UT\, \Sigma$. The restriction of the projection $UT\, \Sigma \rightarrow \Sigma$ from $\mathscr{C}$ to $g^{-1}(\bR P^1)$ is a smooth immersion. The transversality property follows from the local models $\{ \vert z \vert < 1 \, , \; z^n\in \bR \}$.

For the second part of the statement, the function $g$ is monotone on each smooth arc $\alpha \subset g^{-1}(\bR P^1)$. In particular, the image under $g$ of the two extremities of $\alpha$ are necessarily distinct points in $\bR P^1$. It prevents the existence of an arc $\alpha$ as in the statement.
\end{proof}

\begin{proof}[Proof of Proposition \ref{prop:critloc}]
For any small $\delta>0$, we can choose $\eps$ small enough such that $\tgf^{-1}(\bR P^1) \setminus \bR P^1$ is contained in 
\[
\mathscr{B}:= \bigcup_{1\leqslant j\leqslant d} \left\lbrace z\in \bC \, \big| \; \vert z-c_j \vert < \delta  \right\rbrace.
\]
Indeed, for any point $z\in \bC P^1 \setminus \mathscr{B}$, the image $\tgf(z)$ depends continuously on the parameters $\eps_j$. When $\eps_j=0$ for all $1\leqslant j\leqslant d$, the function $\tgf(z)$ satisfies $\tgf^{-1}(\bR P^1)= \bR P^1$ as we have seen above. This proves the claim.

By the Riemann-Hurwitz formula, we showed above that we can also guarantee that there are exactly $4$ simple critical points of $\tgf$ in each neighbourhood $\mathscr{B}_j:= \left\lbrace z\in \bC \, \big| \; \vert z-c_j \vert < \delta  \right\rbrace$, $1\leqslant j\leqslant d$, taking $\eps$ small enough. Moreover, exactly $2$ simple critical points are located on $\bR P^1\cap \mathscr{B}_j$ and the rest of the ramification consists of $2$ simple critical points that are complex conjugate. 

By Lemma \ref{lem:preimage}, the set $\tgf^{-1}(\bR P^1)$ consists of a union of smoothly immersed circles. Since there are exactly $2$ real simple critical points in $\mathscr{B}_j$, there is exactly one circle of $\tgf^{-1}(\bR P^1)$ in $\mathscr{B}_j$ that intersects $\bR P^1$. This circle cannot pass through the remaining critical points in $\mathscr{B}_j$, otherwise it would contain the kind of smooth arc $\alpha$ prohibited by Lemma \ref{lem:preimage}. It follows that this circle is smoothly embedded. Eventually, we claim that there are no extra circles in $\tgf^{-1}(\bR P^1)$. Indeed, any extra circle should be contain in one of the halves of $\mathscr{B}_j\setminus \bR P^1$ for some $1\leqslant j\leqslant d$. By Lemma \ref{lem:preimage}, it cannot pass through the only critical point of $\tgf$ in this half, otherwise it would contain an arc $\alpha$. It is therefore smooth and cover  $\bR P^1$ at least once via $\tgf$. This is in contradiction with the fact that $\tgf$ has only real zeroes.
\end{proof}

For a given parameter $\lambda \in \bR$, define $g(z,w):= f(z+ \lambda i, w)$, that is   
\begin{equation}\label{eq:g}
    g(z,w):=
    w\cdot \prod_{j=1}^d (z-(b_j-\lambda i)) + \prod_{j=1}^{d+1} (z-(a_j-\lambda i))
\end{equation}
where the real parameters $a_1,\cdots,a_{d+1}, b_1, \cdots , b_d$ are as in \eqref{eq:ajbj}. 
The curve $Z(g)$ can be parametrised by
\[
\varrho_\lambda \; :  \;  t\mapsto \Big(t-\lambda i, -\frac{p(t)}{q(t)}\Big)
\]
and the logarithmic Gauss map $\gg$ is given by
\[
\gg(z,w) \;=\; \left[z \cdot \big(p'(z+ \lambda i)+ w\cdot q'(z+ \lambda i)\big)\, : \; w\cdot q(z+ \lambda i) \right]
\]
Let $\tgg$ be the composition of $\gg$ with the parametrisation $\varrho_\lambda$. In the affine chart $\left[u \, : \; 1 \right]$ of $\bC P^1$,
the map $\tgg$ is given by 
\[
\tgg(t)\;=\; (t-\lambda i) \cdot \Big(\dfrac{q'(t)}{q(t)}- \dfrac{p'(t)}{p(t)}\Big)\;=\; \frac{t-\lambda i}{t} \cdot \tgf(t) \;=\; \Big(1-\frac{\lambda i}{t}\Big)\cdot \tgf(t).
\]

\begin{Proposition}\label{prop:deformation}
For $\vert \lambda \vert >0$ and  $\eps>0$ small enough, the preimage $\tgg^{-1}(\bR P^1)$ is smooth and has $d+1$ connected components: one of the connected components contains the points $$\lambda i , a_1, \cdots, a_{d+1}, \infty.$$ 
Each of the $k$ remaining components contains exactly one of the points $b_1, \cdots, b_d$.
\end{Proposition}

Before proving the above statement, let us recall some basic facts from Morse theory. Let $F,G:\bR^2\rightarrow \bR$ be a real analytic functions and let $(x_0,y_0)$ be a non-degenerate Morse critical point of index $1$ of $F$ such that $F(x_0,y_0)=0$. In particular, the zero locus of $F$ at $(x_0,y_0)$ consists of two branches intersecting transversely.
Assume now that $G(x_0,y_0)\neq 0$. Then, it follows from the Morse Lemma with parameters that, for any arbitrarily small $\lambda>0$, the zero locus of $F-\lambda\cdot G$ is smooth at $(x_0,y_0)$. The local deformation $(F-\lambda\cdot G)^{-1}(0)$ of $F^{-1}(0)$ depends on the sign of $G(x_0,y_0)$, see Figure \ref{fig:def}.

Recall as well that the imaginary part of a holomorphic function $h(z)$ has non-degenerate Morse critical point of index $1$ at any simple critical point of $h$. Indeed, up to an appropriate change of coordinate on the source, we can assume that the critical point under consideration is $0$ and that $h$ is given by $h(z)=\alpha+z^2+o(z^2)$ whose Hessian in the coordinate $x:=\Re(z), y:=\Im(z)$ is $\left(\begin{smallmatrix} 0&2\\2&0\end{smallmatrix}\right)$.

\begin{figure}
    \centering
    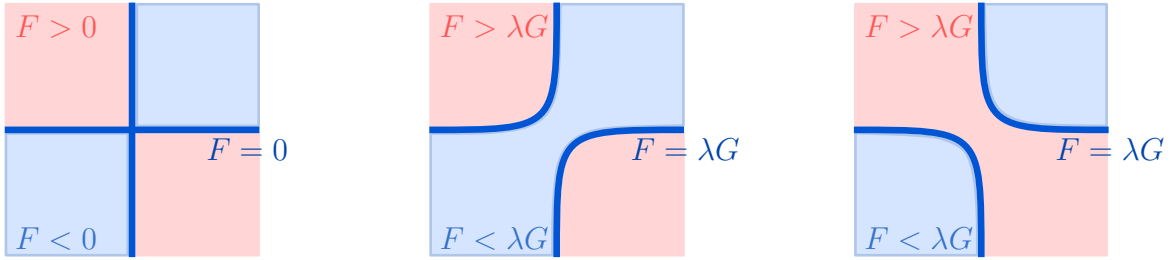
    \caption{The locus $F=0$ (on the left) and its deformations $F=\lambda G$ for $G$ positive (in the middle) and $G$ negative (on the right) and $\lambda>0$.}
    \label{fig:def}
\end{figure}

\begin{proof}
First, observe that every $a_j$ and $b_j$ is a pole of $\tgg$, that $\tgg(i\lambda)=0$ and that $\tgg(\infty)=\tgf(\infty)=-1$. Therefore, each of these points belongs to $\tgg^{-1}(\bR P^1)$.

Write $t=: x+i y$ where $x,y\in \bR$. Then, the locus $\gg^{-1}(\bR P^1)$ is a real algebraic curve in the $(x,y)$-plane that depends algebraically on the parameter $\lambda$. Indeed, the latter locus is the zero set of $\Im\big(\tgf(t)\cdot \overline{\tgf(t)}\big) \in \bR[x,y]$. 
For $\lambda=0$, that is $g=f$, the real algebraic curve $\gg^{-1}(\bR P^1)$ in the $(x,y)$-plane consists of the union of the line $y=0$ with $d$ mutually disjoint smooth ovals, by proposition \ref{prop:critloc}. Each oval intersects $y=0$ in exactly $2$ points and is symmetric with respect to $y=0$. By Remark \ref{rem:critical}, we can localise the latter pair of points for each oval. In particular, the locus $\Im(\tgf)=0$ is as pictured on the top of Figure \ref{fig:defgamma}. Moreover, we have that the imaginary part of $\tgf$ is negative on the complement component of $\Im(\tgf)=0$ that contains $i$ since $\tgf'(0)<0$. Therefore, the sign of $\Im(\tgf)$ is as pictured on top of Figure \ref{fig:defgamma}.

If we denote $F:=\Im(\tgf(x+iy))$ and $G:=\Im\big(\frac{i}{x+iy}\tgf(x+iy)\big)$, we have that $\tgg(x+iy)=(F-\lambda G)(x,y)$. On the $x$-axis, the function $\tgf$ is real valued so that $G(x,0)=x^{-1}\tgf(x)$. Since $\tgf$ is strictly negative at its real critical points, that these critical points are located on $\{x<0\}$, by Proposition \ref{prop:eps}, then $G$ is strictly positive at the critical points of $\tgf$. Therefore, we can determine which of the two possible deformations of Figure \ref{fig:def} applies at each critical point of $\tgf$ to obtain the locus $\Im(\tgg)=0$. The latter locus is picture on the bottom of Figure \ref{fig:defgamma} for $\lambda>0$. The case $\lambda<0$ is obtianed by flipping the latter picture upside down. The result follows.
\end{proof}

\begin{figure}[h]
    \centering
    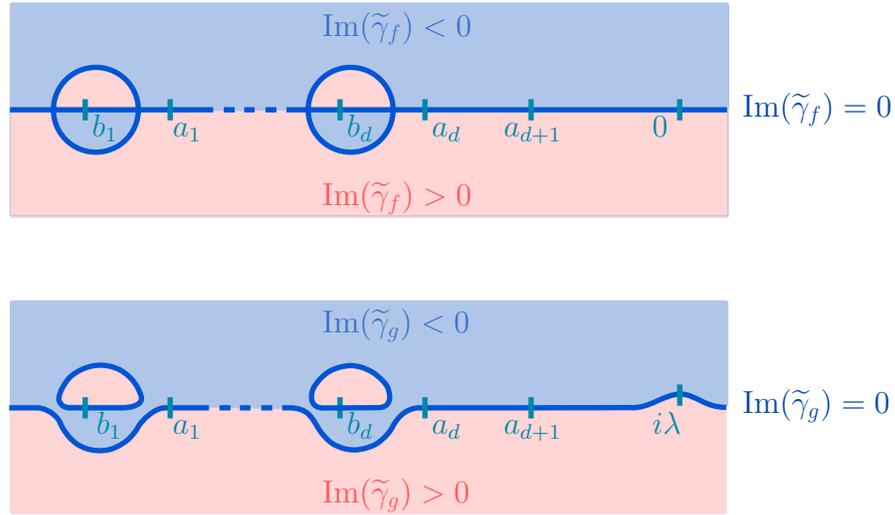
    \caption{The locus $\protect\Im(\protect\tgf)=0$ (on top) in the $(x,y)$-plane and its deformation $\protect\Im(\protect\tgg)=0$ (on the bottom) for small $\lambda>0$.}
    \label{fig:defgamma}
\end{figure}

\section{Proof of Theorem \ref{thm:applicationpatchwork}}\label{sec:proofapplicationpatchwork}

\begin{Lemma}\label{lemma:Harnackdeforme}
For any choice of real parameters $a_1, \cdots ,a_{d+1}, b_1, \cdots, b_d$ as in \eqref{eq:ajbj} and $\lambda \in \bR$ with small enough absolute value, the polynomial 
\begin{equation}\label{eq:h}
    h(z,w):=
    w\cdot \prod_{j=1}^d (z-(-b_j-\lambda i)) + \prod_{j=1}^{d+1} (z-(a_j-\lambda i))
\end{equation}
is such that $\Cr(h)$ is connected.
\end{Lemma}

\begin{proof}
For $\lambda=0$, the curve $Z(h)$ is a simple Harnack curve. To see this, observe that the coordinate $z$ gives a 
parametrisation of $Z(h)$ and that the points of intersection of the closure of $Z(h)$ in $\Sigma_1$ with the toric divisors 
are parametrised by $a_1, \cdots, a_{d+1},
0,-b_1, 
\cdots, b_d, 
-\infty$. In particular, these points are in maximal cyclical 
position according to \cite[Definition 2]{Mikh} and $Z(h)$ is therefore a simple Harnack curve. As a consequence, the 
Log-critical locus $\Cr(h)$ is smooth and connected since $Z(h)$ is rational.
Eventually, the set $\Cr(h)$ remains smooth and connected for small $\vert \lambda \vert$.
\end{proof}

\begin{proof}[Proof of Theorem \ref{thm:applicationpatchwork}]
In order to prove the statement, we will apply Theorem \ref{thm:patchwork} to a specific Viro polynomial $f_t$ constructed from a degree $d$ polynomial $f$ and a $\bZ$-convex function $\nu:\Delta_d\rightarrow \bR$.
 We choose $\nu$ in such a way that the induced subdivision of $\Delta_d$ is the one pictured in Figure \ref{fig:subdiv}. The latter subdivision consists of $d-1$ sub-polygons $T_1, \cdots, T_{d-1}$ defined as $T_j:=\conv \big\lbrace (1,j-1), (1,j), (d-j+1,j-1), (d-j,j) \big\rbrace$ and $d$ sub-polygons $S_1, \cdots, S_{d}$ defined as $S_j:=\conv \big\lbrace (0,j-1), (0,j), (1,j-1), (1,j) \big\rbrace$ for $1\leqslant j \leqslant d-1$ and $S_d:=\conv \big\lbrace (0,d-1), (0,d), (1,d-1) \big\rbrace$. This is clear that such a function $\nu$ exists.

We will now construct the polynomial $f$, up to an irrelevant multiplicative constant. To do so, it suffices to describe for each $P\in \{T_1,\cdots,T_{d-1},S_1, \cdots, S_d\}$ the algebraic curve in $X_P$ defined by the truncation $f^P$. 

We fix an given integer $1\leqslant b \leqslant \left(\begin{smallmatrix} d-1\\2\end{smallmatrix}\right)+1$ as in the statement. Let $k$ denote the integer such that  $\left(\begin{smallmatrix} k-1\\2\end{smallmatrix}\right)+1 < b \leq  \left(\begin{smallmatrix} k\\2\end{smallmatrix}\right)+1$ , $\ell=d-1-k$ and  $r=\ell+\big(\left(\begin{smallmatrix} k-1\\2\end{smallmatrix}\right)+1-b\big)$. 

We now start the construction on $f$. If $\ell\geqslant 1$, we fix $C_{T_1}$ to be the curve defined by any equation $g_1$ of degree $d-1$ as in \eqref{eq:g}. If $\ell\geqslant 2$, we fix $C_{T_2}$ to be the curve defined by an equation $g_2$ of degree $d-2$ where $g_2$ is as in \eqref{eq:g}.  To ensure the compatibility of $C_{T_1}$ with $C_{T_2}$ along the edge $T_1\cap T_2$, we need to take the parameters $a_j$ of $g_2$ to be the parameters $b_j$ of $g_1$. We keep on with the same procedure up to $T_\ell$ included. 

It should be clear by now that we can inductively define the curves $C_{T_i}$, $i> \ell$, using a polynomial as in either \eqref{eq:g} or \eqref{eq:h}, adjusting the parameters $a_j$ of this curve to the parameters $b_j$ of the curve beneath it, using a change of variable $z\mapsto -z$ if necessary. Then, we define any curve $C_{T_i}$, $i\neq r$, using a polynomial as in \eqref{eq:h} and $C_{T_r}$, if $r>\ell$, using a polynomial as in \eqref{eq:g}.

For the curves $C_{S_i}$, $1\leqslant i \leqslant d-1$, we use toric translations $(z,w)\mapsto(\alpha z, \beta w)$ of the simple Harnack curve $H:=\{1+z+zw-w=0\}$. We first choose $C_{S_1}$ to be a toric translation of $H$ compatible with $C_{T_1}$ along $S_1\cap T_1$ (there is a $\bC^*$-family of such). Next, we choose $C_{S_2}$ to be the unique toric translation of $H$
compatible with both $C_{S_1}$ and $C_{T_2}$. We keep on with the same procedure up to $C_{S_{d-1}}$. Eventually, we pick $C_{S_d}$ to be any line compatible with $C_{S_{d-1}}$. Since the curve $C_P$ is smooth with smooth Log-critical locus $\Cr(C_P)$ for any $P\in \{T_1,\cdots,T_{d-1},S_1, \cdots, S_d\}$, the corresponding Viro polynomial $f_t$ is Log non-degenerate.

We now define $C:=\overline{Z(f_t)}\subset \bC P^2$ for $t\in \mathcal{U}$ where $\mathcal{U}$ is the open set provided by Theorem \ref{thm:patchwork}. We claim that $b_0\big(\Cr(C)\big)=b$.
First, consider the patchwork obtained by throwing away the $S_j$-s. For any $1\leqslant j \leqslant \ell$, the Log-critical locus $\Cr(C_{T_j})$ consists of one component intersecting the divisor $\{z=0\}$ and $d-j$ components intersecting $\{w=\infty\}$, according to Lemma \ref{prop:deformation}. After patchworking, each of the $d-j$ components connects to the component of $\Cr(C_{T_{j+1}})$ intersecting $\{z=0\}$, $j=\ell$ included. We obtain $\ell$ components this way. The patchwork of the loci $\Cr(C_{T_j})$, $\ell<j\leqslant r$ provides us with a single component intersecting $\{z=0\}$ and as many compact components in $\ttor$ as there are lattice points in the interior of $T_{\ell+1} \cup \cdots \cup T_r$. Similarly the patchwork of the of the loci $\Cr(C_{T_j})$, $r<j\leqslant d-1$ provides us with a single component intersecting $\{z=0\}$ and as many compact components in $\ttor$ as there are lattice points in the interior of $T_{r+1} \cup \cdots \cup T_{d-1}$. In total, we obtain as many compact components as the number of lattice points in the interior of $T_{\ell+1} \cup \cdots \cup T_{d-1}$ minus the number of lattice points in the interior of the segment $T_r\cap T_{r+1}$, that is  
\[
\textstyle \left(\begin{smallmatrix} d-\ell-2\\2\end{smallmatrix}\right)+1-(d-1-r) = \left(\begin{smallmatrix} k-1\\2\end{smallmatrix}\right)+1-(d-1-r) = b+1-d+\ell.
\]
To conclude, we reintroduce the polygons $S_1, \cdots, S_d$ in the patchwork. All the previous components intersecting $\{z=0\}$ are now connected to each other and form a single connected component of $\Cr(C)$. There are $d-2-\ell$ new compact components in $\ttor$. In total, we obtain $1+(b+1-d+\ell)+(d-2-\ell)=b$ components.
\end{proof}

\begin{figure}[h]
    \centering
    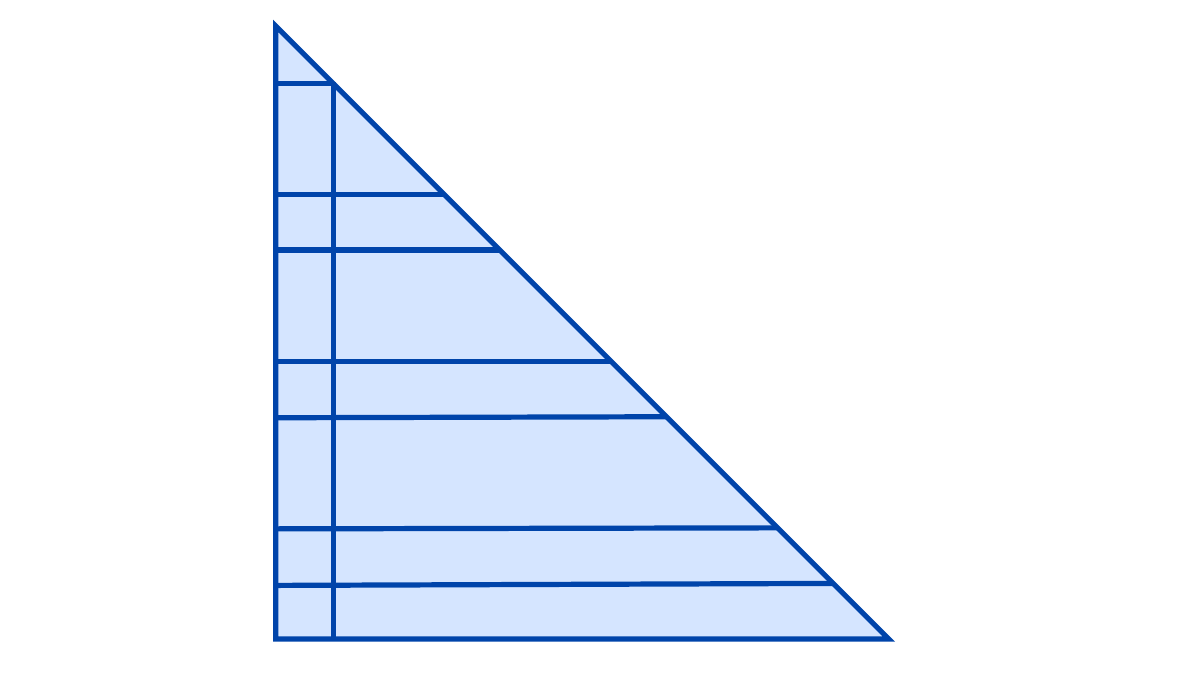
    \caption{A convex subdivision of $\Delta_d$.}
    \label{fig:subdiv}
\end{figure}

\section{Appendix}

In this last section, we illustrate the necessity of the assumption of Theorem \ref{thm:tropicallimit}. Namely, we provide an example of Viro polynomial that is Log degenerate and show that the tropical limit of some Log-inflection points is not as predicted by the latter theorem. Below, we use the notations of Section \ref{sec:proof}.

Consider the Viro polynomial
\[f_t(z,w)=c t +z \big((w-1)(w-1+b)+dz(w-1)+az^2 \big)\]
whose coefficients depends on the complex parameters $a$, $b$, $c$ and $d$. Consider for a moment the polynomial $h(z,w)=(w-1)(w-1+b)+dz(w-1)+az^2$. Since $(w-1)$ is a factor of the truncations of $h$ to $\{j_1=0\}$ and $\{j_1=1\}$ respectively, the point $(0,1)\in V(h)\cap\{z=0\}$ is a singular point of the Log-critical locus $\Cr(h)$, as hinted in the proof of Lemma \ref{lem:newton}(c). In particular, the latter point is a Log-inflection point. Therefore, the Viro polynomial $f_t$ is Log degenerate. However, it is non-degenerate for a generic choice of the coefficients $a$, $b$, $c$ and $d$.

The Riemann-Hurwitz formula implies that there are exactly 3 Log-inflection points in $C_t \cap \mathcal{V}$ for small values of $t$. Repeating similar computations as in Section \ref{sec:proof}, we can provide a parametrisation for the latter points.

The polynomial $\tft(z,\tw)$ is given by 
\[\tft(z,w)=c t +z \big(\tw(\tw+b)+dz\tw+az^2 \big)\]
with Newton polygon $\wD=\conv\big((0,0,1),(1,1,0),(1,2,0),(3,0,0)\big)$. Plainly, the numerator $\wN$ does not depend on $t$ since $\gamma_t$ itself does not depend on $t$. Its Newton polyhedron is therefore contained in the $(j_1,j_2)$-plane and is actually given by $\conv\big((2,0),(5,0),(1,4),(0,4),(0,1)\big)$. The truncation of $\wN$ to the edge $\conv\big((2,0),(0,1)\big)$ is $b^3w-3ab^2z^2$. It follows that $(1,2,3)$ is a tropism of the ideal $\big(\tft,\wN\big)$ since it is orthogonal to the edge $\conv\big((1,1,0),(0,0,1)\big)$ of $\wD$ and the edge $\conv\big((2,0,0),(0,1,0)\big)$ of $\New(\wN)$. The corresponding solutions of the system $\{\tft=\wN=0\}$, which are Log-inflection points of the curve $C_t$, are parametrised by 
\begin{equation}\label{eq:lastparam}
   (\alpha s+o(s), \beta s^2+o(s^2), s^3+o(s^3)) 
\end{equation}
in the $(z,\tw,t)$-coordinates. The composition of this parametrisation with each of $\tft$ and $\wN$ is identically zero. This leads to two equations on the coefficients $\alpha$ and $\beta$ that are eventually equivalent to  
\[\alpha=-\frac{c}{a(3b+1)} \quad \text{ and } \quad \beta= \frac{3au^2}{b}.\]
Using the substitution $t=s^3$ in the solutions \eqref{eq:lastparam} and switching back to the coordinates $(z,w,t)$ gives the following parametrisation for the Log-inflection points
\[(\alpha t^{1/3}+o(t^{1/3}), 1+\beta t^{2/3}+o(t^{2/3}), t+o(t)).\]
Thus, the tropical limit of the 3 corresponding Log-inflection points is the point $(-1/3,0)$, lying on the edge $\conv\big((-1,0),(0,0)\big)$ of the tropical limit. 

Observe that further computations lead to the formula 
\[\gamma_t(z(t),w(t))=\frac{6a}{b}\Big(\frac{-c}{a(3b+1)}t\Big)^{2/3}+o(t^{2/3}).\]
In turn, this formula allows to describe the Log-critical locus of $C_t$ in terms of the Log-critical loci of $c t +z (w-1)(w-1+b)$ and $z \big((w-1)(w-1+b)+dz(w-1)+az^2 \big)$ for $t$ small and generic.

\bibliographystyle{alpha}
\bibliography{Draft}

\end{document}